\documentclass{article}

\usepackage[top=2cm, bottom=2cm, left=3cm, right=3cm]{geometry}

\usepackage{amsmath, amsthm, bm, mathrsfs, mathtools, amssymb, amsfonts}
\usepackage{graphicx}
\usepackage[colorlinks=true, allcolors=blue]{hyperref}
\usepackage{multicol}
\usepackage[all]{xy}
\usepackage{appendix}

\newtheorem{theorem}{Theorem}[section]
\newtheorem{lemma}[theorem]{Lemma}
\newtheorem{proposition}[theorem]{Proposition}
\newtheorem{corollary}[theorem]{Corollary}

\newtheorem{claim}[theorem]{Claim}

\theoremstyle{definition}

\newtheorem{example}{Example}

\theoremstyle{remark}
\newtheorem{remark}{Remark}[section]

\newcommand{\Cov}{\operatorname{Cov}}
\newcommand{\Var}{\operatorname{Var}}

\allowdisplaybreaks

\begin{document}

\title{On the maximal correlation of some stochastic processes}
\author{Yinshan Chang\thanks{Address: College of Mathematics, Sichuan University, Chengdu 610065, China; Email: ychang@scu.edu.cn; Supported by National Key R\&D Program of China (No. 2023YFA1009601)}, Qinwei Chen\thanks{Address: College of Mathematics, Sichuan University, Chengdu 610065, China; Email: qinweic@outlook.com}}

\date{}
\maketitle

\begin{abstract}
We study the maximal correlation coefficient $R(X,Y)$ between two stochastic processes $X$ and $Y$. In the case when $(X,Y)$ is a random walk, we find $R(X,Y)$ using the Cs\'{a}ki-Fischer identity and the lower semicontinuity of the map $\text{Law}(X,Y) \to R(X,Y)$. When $(X,Y)$ is a two-dimensional L\'{e}vy process, we express $R(X,Y)$ in terms of the L\'{e}vy measure of the process and the covariance matrix of the diffusion part of the process. Consequently, for a two-dimensional $\alpha$-stable random vector $(X,Y)$ with $0<\alpha<2$, we express $R(X,Y)$ in terms of $\alpha$ and the spectral measure $\tau$ of the $\alpha$-stable distribution. We also establish analogs and extensions of the Dembo-Kagan-Shepp-Yu inequality and the Madiman-Barron inequality.
\end{abstract}

\section{Introduction}
For two square-integrable non-degenerate real-valued random variables $X$ and $Y$,
the Pearson correlation coefficient $\rho(X,Y)$ is defined by
\[\rho(X,Y)=\frac{\Cov(X,Y)}{\sqrt{\Var(X)}\sqrt{\Var(Y)}}.\]
The correlation coefficient $\rho(X,Y)$
measures the linear dependence of $X$ and $Y$. It is well known that $\rho(X,Y)=0$ if $X$ and $Y$ are independent and square-integrable. However, the converse is not true in general. For two non-degenerate random variables $X$ and $Y$, the maximal correlation coefficient between $X$ and $Y$, introduced by Gebelein \cite{GebeleinMR7220}, is given by
\begin{equation}\label{eq: defn R}
R(X,Y)=\sup\rho(\varphi(X),\psi(Y)),
\end{equation}
where the supremum is over all measurable functions $\varphi$ and $\psi$ such that \[0<E[(\varphi(X))^{2}]<\infty,\quad 0<E[(\psi(Y))^{2}]<\infty.\]
If $X$ or $Y$ is degenerate, we set $R(X,Y)=0$. The quantity $R(X,Y)$ measures the dependence between $X$ and $Y$, with $R(X,Y)=0$ indicating independence, see \cite{SarmanovMR99095}.

According to \cite[Eq. (29)]{RenyiMR0115203}, the maximal correlation coefficient has the following alternative definition: For a random variable $X$, let $L^2_0(X)$ denote the Hilbert space of square-integrable mean $0$ real-valued random variables $Z$ that are measurable with respect to the sigma-field $\sigma(X)$ generated by $X$. For two random variables $X$ and $Y$, the maximal correlation coefficient $R(X,Y)$ is equal to the operator norm of the conditional expectation $\varphi(X)\mapsto E[\varphi(X)|Y]$ from $L^2_0(X)$ to $L^2_0(Y)$.
Equivalently, we have the following expression:
\begin{equation}\label{eq: alt defn R}
R^2(X,Y)=\sup\{E[(E[\varphi(X)|Y])^2]:E[\varphi(X)]=0,E[\varphi^2(X)]=1\}.
\end{equation}

Recently, the maximal correlation coefficient has been generalized by Dadoun and Youssef \cite{DadounYoussefMR4255828} to the context of free probability. However, in this paper, we will concentrate on the classical notion of maximal correlation.

The maximal correlation coefficient plays an important role in various areas of probability and statistics, including information theory \cite{CourtadeMR3638565,MadimanBarronMR2319376}, the hyper-contractivity ribbon of a pair of random variables and the impossibility of non-interactive simulations of joint distributions \cite{AhlswedeGacsMR424401,KamathAnantharam6483335,KamathAnantharamMR3506743}, the optimal transformation for regression \cite{BreimanMR0803258}, the spectral gap of Markov chains and the convergence theory of data augmentation algorithms \cite{LiuMR1279653}.

As noted by R\'{e}nyi in \cite{RenyiMR0115203}, it is often difficult to find the exact value of the maximal correlation coefficient, although it may not seem so at first sight. One reason is that the supremum in \eqref{eq: defn R} generally cannot be replaced by a maximum. The explicit value of the maximal correlation coefficient $R(X,Y)$ is only known in a few cases.
For example, if $(X,Y)$ is jointly Gaussian, then
\begin{equation}\label{eq: jointly Gaussian}
R(X,Y)=|\rho(X,Y)|,
\end{equation}
see \cite{Lancaster}. If $(X,Y)$ is uniformly distributed in the domain $|x|^p+|y|^q\leq 1$ with $p,q>0$, then
\[
R(X,Y)=\frac{1}{\sqrt{(p+1)(q+1)}},
\]
see \cite[Example~5 in Section~2]{CsakiFischerMR0126952}.
The exact value of $R(X,Y)$ was also found in the case when $(X, Y)$ follows a multinomial distribution or multivariate hypergeometric distribution in \cite[Section~6]{CsakiFischerMR0166833}. B\"{u}cher and Staud \cite{BucherStaudMR4835998} obtained the maximal correlation coefficient for the bivariate Marshall-Olkin exponential distribution. There are also several results for maximal correlation coefficients by using orthogonal polynomials, as discussed in \cite{LopezCastanoMR2207171,LopezSalamancaMR3252649,LopezSalamancaMR1622148,NevzorovMR1202799,PapadatosXifaraMR3054093,SzekelyMoriMR792800,TerrellMR704575}. Among these results, \cite{SzekelyMoriMR792800} gives the maximal correlation coefficients for Dirichlet distributions, which provides sharp upper bounds for the maximal correlation coefficients for order statistics studied in \cite{SzekelyMoriMR792800,TerrellMR704575}. In \cite{LopezSalamancaMR1622148,NevzorovMR1202799}, they studied the maximal correlation between the $i$-th and $j$-th records. In \cite{PapadatosXifaraMR3054093}, they provided a unified approach to obtain the maximal correlation coefficient for a subclass of Lancaster distributions.
%The maximal correlation coefficients of certain classes of distributions were studied using various orthogonal polynomials in \cite{LopezCastanoMR2207171,LopezSalamancaMR3252649,LopezSalamancaMR1622148,NevzorovMR1202799,PapadatosXifaraMR3054093,SzekelyMoriMR792800,TerrellMR704575}.
There are also some results on the maximum correlation coefficients for arbitrary distributions. For instance, if $(X_1,Y_1)$ and $(X_2,Y_2)$ are independent, then
\begin{equation*}
    R((X_1,X_2),(Y_1,Y_2)) = \max(R(X_1,Y_1),R(X_2,Y_2)).
\end{equation*}
This is known as the Cs\'{a}ki-Fischer identity \cite[Theorem~6.2]{CsakiFischerMR0166833}, see also \cite[Theorem~1]{WitsenhausenMR363678}. If $X_1,X_2,\ldots,X_n$ are independent and identically distributed (i.i.d.) non-degenerate real-valued
random variables, then for $m\leq n$, we have
\begin{equation}\label{eq: BDKS identity}
R(S_n,S_m)=\sqrt{m/n},
\end{equation}
where $S_k=\sum_{i=1}^{k}X_i$ is the partial sum of $(X_i)_i$, see \cite{BrycDemboKaganMR2141340,DemboKaganSheppMR1828509,NovakMR2089006}. The upper bound
\[R(S_n,S_m)\leq \sqrt{m/n}\]
is known as the Dembo-Kagan-Shepp (DKS) inequality. Yu \cite[Theorem~4.1]{YuMR2422962} further generalized \eqref{eq: BDKS identity} to
\begin{equation}\label{eq: Yu}
R\left(\sum_{i=1}^{m}X_i,\sum_{j=\ell+1}^{n}X_j\right)=\frac{m-\ell}{\sqrt{m(n-\ell)}}
\end{equation}
for $1\leq \ell+1\leq m\leq n$.

In this paper, we are interested in the maximal correlation coefficient between two stochastic processes, e.g., two random walks, two L\'{e}vy processes, or the maximal correlation coefficient between two randomly chosen sub-vectors of a third common random vector. Our main results are
as follows:

\subsection{Maximal correlation coefficients of random walks}
%1. {\bf Maximal correlation coefficients of two-dimensional random walks}.
Let $(\xi_n,\eta_n)_{n\geq 1}$ be i.i.d. random vectors. Let $(S_n,T_n)_{n\geq 0}$ be a random walk starting from $0$ with increments $(\xi_n,\eta_n)$, that is,
\begin{equation}\label{eq: defn random walks}
(S_0,T_0)=0\text{ and for }n\geq 1, (S_n,T_n)=\sum_{m=1}^{n}(\xi_m,\eta_m).
\end{equation}
Let $S=(S_n)_{n\geq 0}$ and $T=(T_n)_{n\geq 0}$.
%Then we have the following result:
Our first main result is
\begin{theorem}\label{thm: random walks}
For any $m\ge 1$,
\[R((S_n)_{n\leq m},(T_n)_{n\leq m})=R(\xi_1,\eta_1)=R(S,T).\]
\end{theorem}
The first equality is a direct consequence of the Cs\'{a}ki-Fischer identity (Theorem~\ref{thm: Csaki-Fisher}) since \[R((S_n)_{n\leq m},(T_n)_{n\leq m})=R((\xi_1,\xi_2,\ldots,\xi_m),(\eta_1,\eta_2,\ldots,\eta_m)).\]
Our main contribution is the proof of the second equality, where we use the lower semicontinuity of $\text{Law}(X,Y)\mapsto R(X,Y)$ (Lemma~\ref{lem: lower semi-continuity}). For self-containedness, we also provide a new probabilistic proof of the Cs\'{a}ki-Fischer identity. As a consequence of Theorem~\ref{thm: random walks}, the central limit theorem and the lower semicontinuity of $\text{Law}(X,Y)\mapsto R(X,Y)$ (Lemma~\ref{lem: lower semi-continuity}), we provide a new proof of \eqref{eq: jointly Gaussian} in Subsection~\ref{subsect: new proof of Lancaster}. This is the main motivation for studying the maximal correlation coefficients for random walks. Similarly, using the law of rare events instead of the central limit theorem, we see that \eqref{eq: jointly Gaussian} holds for bivariate Poisson distributions.

\subsection{Maximal correlation coefficients of two-dimensional L\'{e}vy processes and bivariate stable distributions}
%2. {\bf Maximal correlation coefficients of two-dimensional L\'{e}vy processes and bivariate stable vectors}.
By analyzing maximal correlation coefficients for random walks, we have established the maximal correlation coefficients for bivariate Gaussian vectors. Consequently, it seems logical to investigate the problem of bivariate stable distributions through the analysis of L\'{e}vy processes. We refer to Applebaum's book \cite{ApplebaumMR2512800} for properties of L\'{e}vy processes.
\begin{theorem}\label{thm: Levy}
Let $(X_t,Y_t)$ be a two-dimensional L\'{e}vy process with the characteristic triple $(b,\Sigma,\nu)$. Define
\[\rho=\left\{\begin{array}{ll}
\Sigma_{12}/\sqrt{\Sigma_{11}\Sigma_{22}}, & \Sigma_{11}\Sigma_{22}>0,\\
0, & \Sigma_{11}\Sigma_{22}=0.
\end{array}\right.\]
Let $\mathrm{Op}(\nu)$ be the minimal constant $s$ such that the inequality
\[\int_{\mathbb{R}^2}\varphi(x)\psi(y)\,\nu(\mathrm{d}x,\mathrm{d}y)\leq s\sqrt{\int_{\mathbb{R}^2}(\varphi(x))^2\,\nu(\mathrm{d}x,\mathrm{d}y)\int_{\mathbb{R}^2}(\psi(y))^2\,\nu(\mathrm{d}x,\mathrm{d}y)}\]
holds for all measurable functions $\varphi(x)$ and $\psi(y)$ such that $\varphi(0)=\psi(0)=0$. Then we have that
\[R((X_t)_{t\geq 0},(Y_t)_{t\geq 0})=\max(|\rho|,\mathrm{Op}(\nu)).\]
\end{theorem}
As an application of the above result, we can get the maximal correlation coefficient of a bivariate stable distribution with the stability index $\alpha\in(0,2)$.
\begin{theorem}\label{thm: stable}
Let $(X,Y)\in\mathbb{R}^2$ be a stable random vector with index $\alpha\in(0,2)$. Let $C_{++}$ , $C_{+-}$, $C_{-+}$, $C_{--}$, $D^{x}_{+}$, $D^{x}_{-}$, $D^{y}_{+}$, $D^{y}_{-}$ be some integrals defined in Lemma~\ref{lem: stable Op nu}. Then we have that
\[R(X,Y)=\left\|\begin{pmatrix}
C_{++}/\sqrt{D^{x}_{+}D^{y}_{+}} & \quad C_{+-}/\sqrt{D^{x}_{+}D^{y}_{-}}\\
C_{-+}/\sqrt{D^{x}_{-}D^{y}_{+}} & \quad C_{--}/\sqrt{D^{x}_{-}D^{y}_{-}}
\end{pmatrix}\right\|_2\]
with the convention that $0/0=0$, where $\|\cdot\|_2$ denotes the spectral norm.
\end{theorem}
For random vectors $(\widetilde{X},\widetilde{Y})$ in the domain of attraction of stable laws, Theorem~\ref{thm: stable} provides a lower bound for $R(\widetilde{X},\widetilde{Y})$, see Remark~\ref{rem: doc lower bound} below. For a stable process $(X_t,Y_t)_{t\geq 0}$, coincidentally, we have $R((X_t)_{t\geq 0},(Y_t)_{t\geq 0})=R(X_1,Y_1)$. However, for a general L\'{e}vy process, this equality does not hold, see Examples~\ref{exam: 3} and \ref{exam: 4}.

\subsection{Extensions of the Dembo-Kagan-Shepp-Yu inequality}
%3. {\bf Extensions of the Dembo-Kagan-Shepp-Yu inequality}.
The Dembo-Kagan-Shepp-Yu (DKSY) inequality
\begin{equation}\label{eq: Yu ineq}
R\left(\sum_{i=1}^{m}X_i,\sum_{j=\ell+1}^{n}X_j\right)\leq\frac{m-\ell}{\sqrt{m(n-\ell)}}
\end{equation}
is a sharp inequality for two partial sums of i.i.d. non-degenerate real-valued random variables.
%By taking Gaussian random variables, the above inequality implies \eqref{eq: jointly Gaussian}, that is, the Gebelein-Lancaster theorem. Then it is natural to consider other functions instead of the sum.
It is also natural to consider other functions instead of partial sums. Therefore, we are interested in the maximal correlation coefficient between two randomly chosen subvectors of a common third random vector.
%As a special case, Madiman and Barron \cite{MadimanBarronMR2319376} have found a sharp upper bound for the maximal correlation coefficient between a random vector and its randomly chosen subvector.
Madiman and Barron \cite{MadimanBarronMR2319376} found a sharp upper bound for the maximal correlation coefficient between a random vector and its randomly chosen subvector.

\begin{theorem}[Madiman-Barron]\label{thm: R X (T,X_T)}
Let $X_1,X_2,\ldots,X_n$ be non-degenerate independent random variables taking values in a general measurable space $(F,\mathcal{F})$. Fix a special point $\partial$ outside of $F$. Let $T$ be a random subset of $[n]=\{1,2,\ldots,n\}$. Suppose that $T$ is independent of $(X_1,X_2,\ldots,X_n)$. For $i=1,2,\ldots,n$, define
\[Y_i=\left\{\begin{array}{ll}
X_i, & \text{if }i\in T,\\
\partial, & \text{otherwise}.
\end{array}\right.\]
Then for $n\geq 1$, we have
\[R((X_1,X_2,\ldots,X_n),(Y_1,Y_2,\ldots,Y_n))=
\sqrt{\max\{P(i\in T):i\in[n]\}}.
\]
\end{theorem}
\begin{remark}
The random variables $X_1,X_2,\ldots,X_n$ need not be identically distributed.
\end{remark}
\begin{remark}\label{rem: MB implies DKS}
The Madiman-Barron inequality is a generalization of the Dembo-Kagan-Shepp inequality, see Appendix~\ref{appendix: C}.
\end{remark}
Madiman-Barron \cite{MadimanBarronMR2319376} proved the following Madiman-Barron inequality
\[R((X_1,X_2,\ldots,X_n),(Y_1,Y_2,\ldots,Y_n))\leq
\sqrt{\max\{P(i\in T):i\in[n]\}}.
\]
This upper bound appears in a different form in \cite[Lemmas~2 and 4]{MadimanBarronMR2319376}. In the same paper, this upper bound is further used to deduce generalized entropy power inequalities, the monotonicity of Fisher information, and the Fisher information inequality of Stam. It is an important follow-up work of the breakthrough work \cite{ArtsteinBallBartheNaorMR2083473} by Artstein, Ball, Barthe and Naor on the monotonicity of the Shannon entropy. However, we find no proof in the literature for the lower bound
\[R((X_1,X_2,\ldots,X_n),(Y_1,Y_2,\ldots,Y_n))\geq
\sqrt{\max\{P(i\in T):i\in[n]\}}.
\]
We will give a proof in Subsection~\ref{subsect: proof of MB lower bound}.
For the lower bound, it is crucial that $\partial$ is outside of $F$, see Remark~\ref{rem: partial}. However, for the upper bound, $\partial$ is not necessarily outside of $F$.

Next, we extend Madiman-Barron's result to the case of the maximal correlation coefficient of two randomly chosen subvectors:
\begin{theorem}\label{thm: R (S,X_S) (T,X_T)}
Let $X_1,X_2,\ldots,X_n$ be non-degenerate independent random variables taking values in a general measurable space $(F,\mathcal{F})$. Fix a special point $\partial$ outside of $F$. Let $S$ and $T$ be two random subsets of $[n]=\{1,2,\ldots,n\}$. Suppose that $(S,T)$ is independent of $(X_1,X_2,\ldots,X_n)$. For $i=1,2,\ldots,n$, define
\[Y_i=\left\{\begin{array}{ll}
X_i, & \text{if }i\in S,\\
\partial, & \text{otherwise},
\end{array}\right.\text{ and }Z_i=\left\{\begin{array}{ll}
X_i, & \text{if }i\in T,\\
\partial, & \text{otherwise}.
\end{array}\right.\]
Then for $n\geq 1$, we have that
\begin{equation}\label{eq: R Y Z}
R((Y_1,Y_2,\ldots,Y_n),(Z_1,Z_2,\ldots,Z_n))=\max(R(S,T),\max(r_j:j\in[n])),
\end{equation}
where $r_j$ is the best constant $r$ for the following inequality
\begin{equation}
\sum_{s,t:j\in s\cap t}P(S=s,T=t)\alpha_s\beta_t\leq r\sqrt{\sum_{s:j\in s}P(S=s)\alpha_s^2}\sqrt{\sum_{t:j\in t}P(T=t)\beta_t^2}, \quad \forall \text{ real } \alpha_s, \beta_t.
\end{equation}
In particular, when $S$ and $T$ are independent, we have that $R(S,T)=0$ and that
\[r_{j}=\sqrt{P(j\in S)P(j\in T)}.\]
\end{theorem}
Here, the non-degeneracy of $X_i$ is crucial for the lower bound
\[R((Y_1,Y_2,\ldots,Y_n),(Z_1,Z_2,\ldots,Z_n))\geq \max(r_j:j\in[n]).\]
According to Theorem~\ref{thm: R (S,X_S) (T,X_T)}, in order to calculate $R((Y_1,Y_2,\ldots,Y_n),(Z_1,Z_2,\ldots,Z_n))$, we need to know $R(S,T)$. When $S$ and $T$ depend on each other, we have no good control over $R(S,T)$. However, we have a partial answer in the following special case:
\begin{theorem}\label{thm: R S T}
Let $T$ be a uniform subset of $[n]=\{1,2,\ldots,n\}$ of size $m$. Given $T$, let $S$ be a uniform subset of $T$ with size $k\leq m$. Then we have that
\[R(S,T)=\sqrt{\frac{k(n-m)}{m(n-k)}},\]
with the convention that $0/0=0$.
\end{theorem}
As a corollary, we extend the Dembo-Kagan-Shepp-Yu inequality:
\begin{corollary}[Dembo-Kagan-Shepp-Yu]\label{cor: generalization of Yu}
Let us consider i.i.d. non-degenerate random variables $X_1,X_2,\ldots,X_n$ taking values in a general measurable space $(F,\mathcal{F})$. Then for $1\leq \ell+1\leq m\leq n$, we have that
\begin{equation}\label{eq: generalization of Yu inequality}
R\left(\sum_{i=1}^{m}\delta_{X_i},\sum_{j=\ell+1}^{n}\delta_{X_j}\right)=\frac{m-\ell}{\sqrt{m(n-\ell)}},
\end{equation}
where $\delta_{x}$ is the Dirac measure at $x$.
\end{corollary}
\begin{remark}
The original version of the Dembo-Kagan-Shepp-Yu inequality \eqref{eq: Yu ineq} is stated for real-valued random variables. We believe that a similar result also holds for random vectors by adapting their arguments. The upper bound
\[R\left(\sum_{i=1}^{m}\delta_{X_i},\sum_{j=\ell+1}^{n}\delta_{X_j}\right)\leq\frac{m-\ell}{\sqrt{m(n-\ell)}}\]
in Corollary~\ref{cor: generalization of Yu} is equivalent to the generalized version of DKSY inequality for random vectors. However, the generalized version for random vectors is not clearly written in the literature. So, we decide to state Corollary~\ref{cor: generalization of Yu} and provide a proof for self-containedness. In the present paper, Corollary~\ref{cor: generalization of Yu} is used to deduce Proposition~\ref{prop: Yu replace sum by min}.
\end{remark}
\begin{remark}
In general, the equation \eqref{eq: generalization of Yu inequality} does not hold for independent non-degenerate random variables $X_1$, $X_2$, $\ldots$, $X_n$ with different distributions. Indeed, for some particular distributions, there is a one-to-one correspondence between $\sum_{i=a}^{b}\delta_{X_i}$ and $(X_{a},X_{a+1},\ldots,X_b)$ for all $1\leq a\leq b\leq n$. In such cases, we have that
\[R\left(\sum_{i=1}^{m}\delta_{X_i},\sum_{j=\ell+1}^{n}\delta_{X_j}\right)\geq R(X_{\ell+1},X_{\ell+1})=1.\]
At present, we do not have a satisfactory extension of the Dembo-Kagan-Shepp-Yu inequality for random variables with different distributions. However, we have such an extension for the Dembo-Kagan-Shepp inequality, see Lemma~\ref{lem: calculation of rj special}.
\end{remark}
By applying Theorem~\ref{thm: R (S,X_S) (T,X_T)} and adapting the arguments in \cite{CourtadeMR3638565,MadimanBarronMR2319376}, we can get the following result in information theory:
\begin{theorem}\label{thm: applications in information theory}
Let $X_1,X_2,\ldots,X_n$ be independent real-valued continuous random variables. Let $S\subset T$ be two nested non-empty random subsets of $[n]=\{1,2,\ldots,n\}$, which are independent of $X=(X_1,X_2,\ldots,X_n)$. Let $R$ be the maximal correlation coefficient between $(S,X_S)$ and $(T,X_T)$ given by \eqref{eq: R Y Z}, where $X_{S}=(X_i)_{i\in S}$ and $X_{T}=(X_j)_{j\in T}$. Denote by $I(Z)=\int_{-\infty}^{\infty}\frac{(f'(z))^2}{f(z)}\,\mathrm{d}z$ the Fisher information of the continuous random variable $Z$ with the density $f(z)$. Then we have that
\begin{equation}\label{eq: Fisher information inequality}
\sum_{t\subset[n]}P(T=t)I(\sum_{j\in t}X_j)\mu_t^2\leq R^2\sum_{s\subset[n]}P(S=s)I(\sum_{i\in s}X_i)\lambda_s^2,
\end{equation}
where $\mu_t=\sum_{s\subset[n]}P(S=s|T=t)\lambda_s$ for $t\subset[n]$ and $(\lambda_s)_{s\subset[n]}$ is an arbitrary real vector.
\end{theorem}

\emph{Organization of the paper}: In Section~\ref{sect: preliminaries}, we present several useful properties of maximal correlation coefficients. The proofs of Lemma~\ref{lem: submultiplicative} and Lemma~\ref{lem: lower semi-continuity} are postponed to the appendix. In Section~\ref{sect: random walks}, we prove Theorem~\ref{thm: random walks} and give a new proof of Lancaster's classical result \eqref{eq: jointly Gaussian}. In Section~\ref{sect: Levy and stable}, we first prove Theorem~\ref{thm: Levy}, and then use Theorem~\ref{thm: Levy} to prove Theorem~\ref{thm: stable}. We also give several examples in this section. In Section~\ref{sect: MB like}, we prove the lower bounds of Theorem~\ref{thm: R X (T,X_T)}, Theorem~\ref{thm: R (S,X_S) (T,X_T)}, Theorem~\ref{thm: R S T}, Corollary~\ref{cor: generalization of Yu} and Theorem~\ref{thm: applications in information theory} in separate subsections.
We also give a Dembo-Kagan-Shepp-Yu-type upper bound in Proposition~\ref{prop: Yu replace sum by min}, where the summation in \eqref{eq: Yu ineq} is replaced by the minimum in \eqref{eq: Yu replace sum by min}.
Finally, we present two open problems in Section~\ref{sect: open problems}.

\section{Preliminaries}\label{sect: preliminaries}

\subsection{Basic properties}
We collect several basic properties as follows:
\begin{enumerate}
\item $R(X,Y)\in[0,1]$.
\item $R(X,Y)=0$ if and only if $X$ and $Y$ are independent.
\item $R(X,Y)=R(Y,X)$.
\item In general, the supremum in the definition of the maximal correlation coefficient cannot be replaced by the maximum. For example, suppose that $M$ and $N$ are independent Poisson random variables of parameter $1$. Let $X=M-N$ and $Y=M-2N$. Then we have that $R(X,Y)=1$. Indeed, $\lim_{m\to+\infty}\rho(1_{X=m},1_{Y=m})=1$. However, for any measurable functions $\varphi$ and $\psi$ satisfying $E[\varphi(X)]=E[\psi(Y)]=0$ and $\Var(\varphi(X))=\Var(\psi(Y))=1$, we have that $\rho(\varphi(X),\psi(Y))<1$. R\'{e}nyi \cite[Theorems 1 and 2]{RenyiMR0115203} established sufficient conditions for the supremum to be attained.
\item If $U=\varphi(X)$ is a measurable function of $X$ and $V=\psi(Y)$ is a measurable function of $Y$, then $R(U,V)\leq R(X,Y)$.
\end{enumerate}

\subsection{Submultiplicative property}

\begin{lemma}[Lemma~2.1 of \cite{YuMR2422962}]\label{lem: submultiplicative}
Suppose that $X$ and $Z$ are conditionally independent given $Y$. Then
\[R(X,Z)\leq R(X,Y)R(Y,Z).\]
\end{lemma}
This is the first part of \cite[Lemma~2.1]{YuMR2422962}. For self-containedness, we provide a short proof in Appendix~\ref{appendix: A}.

\subsection{Cs\'{a}ki-Fischer identity}

The following result is known as the Cs\'{a}ki-Fischer identity \cite[Theorem~6.2]{CsakiFischerMR0166833}, see also \cite[Theorem~1]{WitsenhausenMR363678}.
\begin{theorem}[Cs\'{a}ki-Fischer identity]\label{thm: Csaki-Fisher}
Suppose that $(X_1,Y_1)$ and $(X_2,Y_2)$ are independent. Then we have that
\begin{equation*}
    R((X_1,X_2),(Y_1,Y_2)) = \max(R(X_1,Y_1),R(X_2,Y_2)).
\end{equation*}
\end{theorem}

We provide a simple probabilistic proof for self-containedness.

\begin{proof}
Let $\mathcal{F}=\sigma(X_1,Y_1)$. Note that
\begin{align*}
\mathrm{Cov}(f(X_1,X_2),g(Y_1,Y_2))&=E[\mathrm{Cov}(f(X_1,X_2),g(Y_1,Y_2)|\mathcal{F})]\\
&\quad+\mathrm{Cov}(E[f(X_1,X_2)|\mathcal{F}],E[g(Y_1,Y_2)|\mathcal{F}]),
\end{align*}
where the conditional covariance $\mathrm{Cov}(U,V|\mathcal{F})$ of two random variables $U,V$ given the sigma-field $\mathcal{F}$ is defined by
\begin{equation*}
\mathrm{Cov}(U,V|\mathcal{F})=E[UV|\mathcal{F}]-E[U|\mathcal{F}]E[V|\mathcal{F}].
\end{equation*}
To simplify the notation, let
\begin{equation*}
I_1=E[\mathrm{Cov}(f(X_1,X_2),g(Y_1,Y_2)|\mathcal{F})], I_2=\mathrm{Cov}(E[f(X_1,X_2)|\mathcal{F}],E[g(Y_1,Y_2)|\mathcal{F}]).
\end{equation*}
Due to the conditional independence of $X_2$ and $Y_1$ given $X_1$, $E[f(X_1,X_2)|\mathcal{F}]$ is a function of $X_1$. Similarly, $E[g(Y_1,Y_2)|\mathcal{F}]$ is a function of $Y_1$. By the definition of maximal correlation coefficients, we have that
\begin{equation*}
|I_2|\leq R(X_1,Y_1)\sqrt{\operatorname{Var}(E[f(X_1,X_2)|\mathcal{F}])\operatorname{Var}(E[g(Y_1,Y_2)|\mathcal{F}])}.
\end{equation*}
Next, we look for the upper bound of $I_1$. By the definition of maximal correlation coefficients and independence between $(X_1,Y_1)$ and $(X_2,Y_2)$, we have that
\begin{equation*}
|\operatorname{Cov}(f(X_1,X_2),g(Y_1,Y_2)|\mathcal{F})|\leq R(X_2,Y_2)\sqrt{\operatorname{Var}(f(X_1,X_2)|\mathcal{F})\operatorname{Var}(g(Y_1,Y_2)|\mathcal{F})}.
\end{equation*}
By taking the expectation on both sides and using the Cauchy-Schwarz inequality, we obtain that
\[\begin{aligned}
|I_1|&\leq E\left[R(X_2,Y_2)\sqrt{\operatorname{Var}(f(X_1,X_2)|\mathcal{F})\operatorname{Var}(g(Y_1,Y_2)|\mathcal{F})} \right]\\
&\leq R(X_2,Y_2)\sqrt{E[\operatorname{Var}(f(X_1,X_2)|\mathcal{F})]E[\operatorname{Var}(g(Y_1,Y_2)|\mathcal{F})]}.
\end{aligned}\]
For simplicity of notation, let
\begin{equation*}
    A_{f}^{X}=E[\operatorname{Var}(f(X_1,X_2)|\mathcal{F})],
    A_{g}^{Y}=E[\operatorname{Var}(g(Y_1,Y_2)|\mathcal{F})],
\end{equation*}
\begin{equation*}
    B_{f}^{X}=\operatorname{Var}(E[f(X_1,X_2)|\mathcal{F}]),
    B_{g}^{Y}=\operatorname{Var}(E[g(Y_1,Y_2)|\mathcal{F}]).
\end{equation*}
Note that $A_{f}^{X}+B_{f}^{X}=\operatorname{Var}(f(X_1,X_2))$, $A_{g}^{Y}+B_{g}^{Y}=\operatorname{Var}(g(Y_1,Y_2))$. Then we have that
\[\begin{aligned}
|\operatorname{Cov}&(f(X_1,X_2),g(Y_1,Y_2))|\\
&\leq|I_{1}|+|I_{2}|\\
&\leq R(X_1,Y_1)\sqrt{B_{f}^{X}B_{g}^{Y}}+R (X_2,Y_2)\sqrt{A_{f}^{X}A_{g}^{Y}}\\
&\leq\max(R(X_1,Y_1),R(X_2,Y_2))(\sqrt{A_{f}^{X}A_{g}^{Y}}+\sqrt{B_{f}^{X}B_{g}^{Y}}) \\
&\leq\max(R(X_1,Y_1),R (X_2,Y_2))\sqrt{\sqrt{A_{f}^{X}}^{2}+\sqrt{B_{f}^{X}}^{2}}\sqrt{\sqrt{A_{g}^{Y}}^{2}+\sqrt{B_{g}^{Y}}^{2}} \\
&=\max(R(X_1,Y_1),R(X_2,Y_2))\sqrt{\mathrm{Var}(f(X_1,X_2))\mathrm{Var}(g(Y_1,Y_2))}.
\end{aligned}\]
Hence, $R(( X_1,X_2 ),(Y_1,Y_2))\leq\max(R(X_1,Y_1),R(X_2,Y_2))$.

On the other hand, we have that
\[R(( X_1,X_2 ),(Y_1,Y_2))\ge\max(R(X_1,Y_1),R(X_2,Y_2)).\]
\end{proof}
For the Cs\'{a}ki-Fisher identity, the independence between the random vectors $(X_1,Y_1)$ and $(X_2,Y_2)$ is necessary. We will show this by the following example:
\begin{example}
Let $X_1,Y_1$ and $\sigma$ be independent random variables with the same distribution such that $P(\sigma=-1)=P(\sigma=1)=1/2$. Since $X_1$ is independent of $Y_1$, we have $R(X_1,Y_1)=0$. Define
\[X_2=\sigma Y_1,Y_2=\sigma X_1.\]
Then the joint distribution of $X_2,Y_2$ is given in Table~\ref{table: joint distribution of X2 Y2}.
\begin{table}
    \centering
    \begin{tabular}{|c|c|c|}
    \hline
        $X_2\backslash Y_2$ & $-1$ & $1$ \\ \hline
        $-1$ & $1/4$ & $1/4$ \\ \hline
        $1$ & $1/4$ & $1/4$ \\ \hline
    \end{tabular}
\caption{Joint distribution of $(X_2,Y_2)$}
\label{table: joint distribution of X2 Y2}
\end{table}

Immediately, we see that $X_2$ is independent of $Y_2$. Hence, we have that $R(X_2,Y_2)=0$. Note that $X_1X_2=Y_1Y_2=\sigma X_1Y_1$, and they are non-degenerate. Consequently, we have that
\[R((X_1,X_2),(Y_1,Y_2 ))\geq R(X_1X_2,Y_1Y_2)=1>0=\max(R(X_1,Y_1),R(X_2,Y_2)).\]
\end{example}

\subsection{Lower semi-continuity}

The maximal correlation coefficient $R(X,Y)$ is fully determined by the joint distribution $\mu$ of $(X,Y)$. Therefore, for a joint distribution $\mu$ of $(X,Y)$, its maximal correlation coefficient $R(\mu)$ is well-defined by $R(\mu)=R(X,Y)$. Suppose that $(X,Y)$ takes values in a product space $S=\mathcal{X}\times\mathcal{Y}$ of Polish spaces $\mathcal{X}$ and $\mathcal{Y}$ with the Borel $\sigma$-field $\mathcal{S}$. Let $\mathcal{P}$ be the space of probability measures on $(S,\mathcal{S})$. We equip $\mathcal{P}$ with the L\'{e}vy-Prokhorov metric $\pi$, see \cite[Eq.~(6.10)]{BillingsleyMR1700749} for the definition of this metric. By \cite[Theorem~6.8]{BillingsleyMR1700749}, $(\mathcal{P},\pi)$ is Polish, and weak convergence is equivalent to convergence with respect to the metric $\pi$. In this section, we will show the lower semicontinuity of $\mu\mapsto R(\mu)$ with respect to the metric $\pi$. It is a generalization of \cite[Theorem~1]{BrycDemboKaganMR2141340}.
\begin{lemma}\label{lem: lower semi-continuity}
Suppose that $(X,Y)$ takes values in a product space $S=\mathcal{X}\times\mathcal{Y}$ of Polish spaces $\mathcal{X}$ and $\mathcal{Y}$ with the Borel $\sigma$-field $\mathcal{S}$ and let $\mu$ be the joint distribution of $(X,Y)$. As a function on the Polish space $(\mathcal{P},\pi)$, $\mu\mapsto R(\mu)=R(X,Y)$ is lower semicontinuous.
\end{lemma}
As noted by an anonymous referee, the semicontinuity is quite natural since the supremum operator preserves semicontinuity. However, for self-containedness, we provide a proof in Appendix~\ref{appendix: B}.
\begin{remark}
After proving Lemma~\ref{lem: lower semi-continuity}, we became aware of a very similar statement in \cite[Section~II-A]{KamathAnantharamMR3506743}. While the proof ideas are similar, their result has restrictions. Firstly, $X$ and $Y$ are restricted to finite state spaces. Secondly, the convergence of $(X_n,Y_n)$ towards $(X,Y)$ is assumed to be the convergence with respect to the total variation distance instead of the weak convergence.
\end{remark}

\section{Random walks}\label{sect: random walks}

\subsection{Proof of Theorem~\ref{thm: random walks}}
In this subsection, we prove Theorem~\ref{thm: random walks}. As we have explained, the first equality is a direct consequence of the Cs\'{a}ki-Fischer identity, see Theorem~\ref{thm: Csaki-Fisher}. To prove the second equality, for natural numbers $m\geq 1$, we define two continuous-time processes
\[S^{(m)}_t=S_{\min(t,m)}\text{ and }T^{(m)}_t=T_{\min(t,m)},t\geq 0.\]
Then we have that
\[R(S^{(m)},T^{(m)})=R((S_n)_{n\leq m},(T_n)_{n\leq m})=R(\xi_1,\eta_1).\]
As $m\to\infty$, the process $(S^{(m)},T^{(m)})$ converges to $(S,T)$ in the Skorokhod space $D_{\mathbb{R}^2}[0,\infty)$ by \cite[Theorem~1.14]{JacodShiryaevMR1943877}, where $D_{\mathbb{R}^2}[0,\infty)$ denotes the space of c\`{a}dl\`{a}g functions $\omega:[0,\infty)\to\mathbb{R}^2$ endowed with the Skorokhod topology. Therefore, by Lemma~\ref{lem: lower semi-continuity}, we have that
\[R(S,T)\leq\liminf_{m\to\infty}R(S^{(m)},T^{(m)})=R(\xi_1,\eta_1).\]
Finally, since $\xi_1=S_1-S_0$ is a measurable function of $S$ and $\eta_1=T_1-T_0$ is a measurable function of $T$, we have the opposite inequality
\[R(\xi_1,\eta_1)\leq R(S,T).\]

\subsection{A new proof of \texorpdfstring{\eqref{eq: jointly Gaussian}}{(3)}}\label{subsect: new proof of Lancaster}
In this subsection, we provide a new proof of \eqref{eq: jointly Gaussian} based on Theorem~\ref{thm: random walks}, the central limit theorem and Lemma~\ref{lem: lower semi-continuity}.

Let $(X, Y)$ be a two-dimensional Gaussian vector with covariance matrix $\Sigma$ and Pearson correlation coefficient $r$.

Consider a random vector $\left(\xi,\eta\right)$ with the joint distribution in Table~\ref{table: joint distribution}:
\begin{table}
\centering
\begin{tabular}{|c|c|c|}
\hline
$\xi\backslash\eta$ & $c$ & $d$ \\
\hline
$a$ & $p_{ac}$ & $p_{ad}$ \\
\hline
$b$ & $p_{bc}$ & $p_{bd}$ \\
\hline
\end{tabular}
\caption{Joint distribution of $(\xi,\eta)$}
\label{table: joint distribution}
\end{table}

Then for a non-constant function $f$ on $\{a,b\} $ and a non-constant function $g$ on $\{c,d\}$, the Pearson correlation coefficient between $f(\xi)$ and $g(\eta)$ is given by
\begin{equation*}
\operatorname{sgn}(f(a)-f(b))\operatorname{sgn}(g(c)-g(d))\det\begin{pmatrix}p_{ac}&p_{ad}\\p_{bc}&p_{bd}\end{pmatrix}/\sqrt{p_ap_bp_cp_d},
\end{equation*}
where $p_a=P(\xi=a)$, $p_b=P(\xi=b)$, $p_{c}=P(\eta=c)$ and $p_d=P(\eta=d)$. Consequently,
\begin{equation}\label{eq: bivariate Bernoulli}
R(\xi,\eta)=\left |\rho(\xi,\eta)\right |=\left.\left|\det\begin{pmatrix}p_{ac}&p_{ad}\\p_{bc}&p_{bd}\end{pmatrix}\right.\right|/\sqrt{p_ap_bp_cp_d}.
\end{equation}
By appropriately choosing $a$, $b$, $c$, $d$ and the probability mass function $p$, we can ensure that $(\xi, \eta)$ has covariance matrix $\Sigma$ and Pearson correlation coefficient $r$.

Let $(\xi_n,\eta_n)_{n\geq 1}$ be i.i.d. random vectors such that $(\xi_n,\eta_n)$ has the same distribution as $(\xi,\eta)$. Define a two-dimensional random walk $(S,T)$ as in \eqref{eq: defn random walks}. By Theorem~\ref{thm: random walks}, we have that
\[R\left(\frac{S_m-ES_m}{\sqrt{m}},\frac{T_m-ET_m}{\sqrt{m}}\right)\leq R(S,T)=R(\xi_1,\eta_1)=|r|.\]
By the central limit theorem, $\left(\frac{S_m-ES_m}{\sqrt{m}},\frac{T_m-ET_m}{\sqrt{m}}\right)$ converges to $(X,Y)$ in distribution as $m\to\infty$. By Lemma~\ref{lem: lower semi-continuity}, we see that
\[R(X,Y)\leq\liminf_{m\to\infty}R\left(\frac{S_m-ES_m}{\sqrt{m}},\frac{T_m-ET_m}{\sqrt{m}}\right)\leq |r|=|\rho(X,Y)|.\]
On the other hand, by definition, we have that $R(X,Y)\geq |\rho(X,Y)|$.
Therefore, \eqref{eq: jointly Gaussian} holds.

\section{Two-dimensional L\'{e}vy processes and bivariate stable distributions}\label{sect: Levy and stable}
Recall the L\'{e}vy-Khinchine formula for a two-dimensional L\'{e}vy process $(X_t,Y_t)_{t\geq 0}$:
\begin{align}\label{eq: Levy-Khinchine formula}
E[&e^{i(u_1X_t+u_2Y_t)}]=\exp\left(t\left\{i(b_1u_1+b_2u_2)-\frac{1}{2}\sum_{j,k=1,2}\Sigma_{jk}u_ju_k\right.\right.\notag\\
&+\left.\left.\int_{\mathbb{R}^2\setminus\{(0,0)\}}\left(e^{i(u_1x+u_2y)}-1-i(u_1x+u_2y)I_{(0,1]}(x^2+y^2)\right)\,\nu(\mathrm{d}x,\mathrm{d}y)\right\}\right).
\end{align}
Here, $\Sigma=(\Sigma_{jk})_{j,k=1,2}$ is the covariance matrix of the Brownian motion in the L\'{e}vy-It\^{o} decomposition, and $\nu(\mathrm{d}x,\mathrm{d}y)$ is the L\'{e}vy measure on $\mathbb{R}^2\setminus\{(0,0)\}$ satisfying
\[\int_{\mathbb{R}^2\setminus\{(0,0)\}}\min(1,x^2+y^2)\,\nu(\mathrm{d}x,\mathrm{d}y)<\infty.\]
A similar formula holds for an infinitely divisible random vector. For a two-dimensional stable random vector, the L\'{e}vy measure $\nu$ takes the form
\begin{equation}\label{eq: stable Levy measure spectral measure}
\nu(B)=\int_{0}^{\infty}\int_{\mathbb{R}/(2\pi\mathbb{Z})}I_{B}(r\cos\theta,r\sin\theta)\frac{1}{r^{1+\alpha}}\,\mathrm{d}r\tau(\mathrm{d}\theta),\quad B\in\mathcal{B}(\mathbb{R}^2\setminus\{(0,0)\}),
\end{equation}
where the spectral measure $\tau(\mathrm{d}\theta)$ is a finite Borel measure and $\mathcal{B}(\mathbb{R}^2\setminus\{(0,0)\})$ denotes the Borel $\sigma$-field of $\mathbb{R}^2\setminus\{(0,0)\}$.

\subsection{Proof of Theorem~\ref{thm: Levy}}
In this subsection, we prove Theorem~\ref{thm: Levy}. The proof relies crucially on the L\'{e}vy-It\^{o} decomposition and the martingale representation theorem for L\'{e}vy processes. We will briefly present these classical results. We focus on two-dimensional L\'{e}vy processes. Without loss of generality, we may assume that the L\'{e}vy process is canonical. Take the Skorokhod space $D_{\mathbb{R}^2}[0,\infty)$ as the sample space $\Omega$, where $D_{\mathbb{R}^2}[0,\infty)$ is the space of c\`{a}dl\`{a}g functions $\omega:[0,\infty)\to\mathbb{R}^2$ endowed with the Skorokhod topology and the corresponding Borel $\sigma$-field. For $\omega=(\omega_t)_{t\geq 0}$, define $(X_t,Y_t)=\omega_t$ for $t\geq 0$. Let $P$ be a probability measure on $\Omega$ such that $(X_t,Y_t)_{t\geq 0}$ is a two-dimensional L\'{e}vy process. Let
\[\Delta X_t=X_t-\lim_{s\uparrow t}X_s\]
be the jump of $(X_t)_{t\geq 0}$ at time $t$. Similarly, we define $\Delta Y_t$. Then the set $\{t\geq 0:\Delta X_t\neq 0\text{ or }\Delta Y_t\neq 0\}$ of jumping times is countable.
\begin{theorem}[L\'{e}vy-It\^{o} decomposition]
If $(X_t,Y_t)_{t\geq 0}$ is a two-dimensional L\'{e}vy process, then there exists $b=(b^{X},b^{Y})\in\mathbb{R}^2$, a two-dimensional Brownian motion $(B^{X}_t,B^{Y}_t)_{t\geq 0}$ with the covariance matrix $\Sigma$ and a Poisson random measure $N(\mathrm{d}t,\mathrm{d}x,\mathrm{d}y)$ with the intensity measure $\mathrm{d}t\nu(\mathrm{d}x,\mathrm{d}y)$ such that for $t\geq 0$,
\begin{equation}\label{eq: Levy-Ito decomposition}
\begin{aligned}
X_t=b^{X}t+B^{X}_t+\int_{0}^{t}\int_{|x|<1}x\widetilde{N}^{X}(\mathrm{d}s,\mathrm{d}x)+\int_{0}^{t}\int_{|x|\geq 1}xN^{X}(\mathrm{d}s,\mathrm{d}x),\\
Y_t=b^{Y}t+B^{Y}_t+\int_{0}^{t}\int_{|y|<1}y\widetilde{N}^{Y}(\mathrm{d}s,\mathrm{d}y)+\int_{0}^{t}\int_{|y|\geq 1}yN^{Y}(\mathrm{d}s,\mathrm{d}y),
\end{aligned}
\end{equation}
where
\[N^{X}(\mathrm{d}t,\mathrm{d}x)=1_{x\neq 0}\int_{\mathbb{R}}N(\mathrm{d}t,\mathrm{d}x,\mathrm{d}y), N^{Y}(\mathrm{d}t,\mathrm{d}y)=1_{y\neq 0}\int_{\mathbb{R}}N(\mathrm{d}t,\mathrm{d}x,\mathrm{d}y),\] \[\nu^{X}(\mathrm{d}x)=1_{x\neq 0}\int_{\mathbb{R}}\nu(\mathrm{d}x,\mathrm{d}y), \nu^{Y}(\mathrm{d}y)=1_{y\neq 0}\int_{\mathbb{R}}\nu(\mathrm{d}x,\mathrm{d}y)\]
and
\[\widetilde{N}^{X}(\mathrm{d}t,\mathrm{d}x)=N^{X}(\mathrm{d}t,\mathrm{d}x)-\mathrm{d}t\nu^{X}(\mathrm{d}x), \widetilde{N}^{Y}(\mathrm{d}t,\mathrm{d}y)=N^{Y}(\mathrm{d}t,\mathrm{d}y)-\mathrm{d}t\nu^{Y}(\mathrm{d}y).\]
Moreover, $(B_t)_{t\geq 0}$ is independent of $N(\mathrm{d}t,\mathrm{d}x,\mathrm{d}y)$ and
\begin{equation*}
N(\mathrm{d}t,\mathrm{d}x,\mathrm{d}y)=\sum_{(s,\Delta X_s,\Delta Y_s):(\Delta X_s,\Delta Y_s)\neq (0,0)}\delta_{(s,\Delta X_s,\Delta Y_s)}(\mathrm{d}t,\mathrm{d}x,\mathrm{d}y),
\end{equation*}
where $\delta_{z_0}(\mathrm{d}z)$ is the Dirac measure at $z_0\in[0,\infty)\times(\mathbb{R}^2\setminus\{(0,0)\})$. In addition, $(B^{X}_{t})_{t\geq 0}$ is a measurable function of $(X_t)_{t\geq 0}$, and $(B^{Y}_t)_{t\geq 0}$ is a measurable function of $(Y_t)_{t\geq 0}$.
\end{theorem}
We refer to \cite[Theorem~2.4.16]{ApplebaumMR2512800} for a proof.

Next, we introduce the martingale representation. Fix $T>0$. Let $\mathcal{F}^{X}_{T}$ be the augmentation of $\sigma(X_t,0\leq t\leq T)$.
\begin{theorem}[The It\^{o} representation]\label{thm: martingale representation}
Let $F$ be a square-integrable $\mathcal{F}^{X}_{T}$-measurable random variable. Then there exist square-integrable predictable processes $\varphi^{X}$ and $\psi^{X}$ such that almost surely,
\begin{align*}\label{eq: integral representation of F}
F&=E[F]+\int_{0}^{T}\varphi^{X}(t)\,\mathrm{d}B^{X}_t+\int_{0}^{T}\int_{\mathbb{R}\setminus\{0\}}\psi^{X}(t,x)\widetilde{N}^{X}(\mathrm{d}t,\mathrm{d}x)\notag\\
&=E[F]+\int_{0}^{T}\varphi^{X}(t)\,\mathrm{d}B^{X}_t+\int_{0}^{T}\int_{\mathbb{R}\setminus\{0\}}\int_{\mathbb{R}}\psi^{X}(t,x)\widetilde{N}(\mathrm{d}t,\mathrm{d}x,\mathrm{d}y).
\end{align*}
\end{theorem}
Clearly, similar results hold for $G\in L^2(\Omega,\mathcal{F}^{Y}_T,P)$. We refer to \cite[Section~5.3]{ApplebaumMR2512800} for details. One way to prove the representation theorem is to use Wiener-L\'{e}vy chaos decomposition, see \cite[Theorem~2]{ItoMR0077017} and \cite[Theorem~1.1]{KunitaMR2083711}.

Now, let us start to prove Theorem~\ref{thm: Levy}. Since $(B^{X}_t)_{t\geq 0}$ is a measurable function of $(X_t)_{t\geq 0}$, $B^{X}_{1}$ is a square-integrable measurable function of $(X_t)_{t\geq 0}$. Similarly, $B^{Y}_{1}$ is a square-integrable measurable function of $(Y_t)_{t\geq 0}$. Hence, we have that
\begin{equation}\label{eq: R geq |rho|}
R((X_t)_{t\geq 0},(Y_t)_{t\geq 0})\geq\max(\rho(B^{X}_1,B^{Y}_1),\rho(-B^{X}_1,B^{Y}_1))=|\rho|.
\end{equation}
Next, by the definition of $\mathrm{Op}(\nu)$, for each small $\varepsilon>0$, there exist measurable functions $\varphi(x)$ and $\psi(y)$ such that $\varphi(0)=\psi(0)=0$, $0<\int_{\mathbb{R}^2}(\varphi(x))^2\,\nu(\mathrm{d}x,\mathrm{d}y)\int_{\mathbb{R}^2}(\psi(y))^2\,\nu(\mathrm{d}x,\mathrm{d}y)<\infty$ and that
\[\int_{\mathbb{R}^2}\varphi(x)\psi(y)\,\nu(\mathrm{d}x,\mathrm{d}y)\geq (\mathrm{Op}(\nu)-\varepsilon)\sqrt{\int_{\mathbb{R}^2}(\varphi(x))^2\,\nu(\mathrm{d}x,\mathrm{d}y)\int_{\mathbb{R}^2}(\psi(y))^2\,\nu(\mathrm{d}x,\mathrm{d}y)}.\]
Let
\begin{align*}
F&=\int_{0}^{1}\int_{\mathbb{R}\setminus\{0\}}\varphi(x)\widetilde{N}^{X}(\mathrm{d}t,\mathrm{d}x)\\ &=\lim_{n\to\infty}\sum_{(t,\Delta X_t):t\leq 1,|\Delta X_t|>1/n}\varphi(\Delta X_t)-\int_{\mathbb{R}^2}1_{|x|>1/n}\varphi(x)\,\nu(\mathrm{d}x,\mathrm{d}y),
\end{align*}
\begin{align*}
G&=\int_{0}^{1}\int_{\mathbb{R}\setminus\{0\}}\psi(y)\widetilde{N}^{Y}(\mathrm{d}t,\mathrm{d}y)\\ &=\lim_{n\to\infty}\sum_{(t,\Delta Y_t):t\leq 1,|\Delta Y_t|>1/n}\psi(\Delta Y_t)-\int_{\mathbb{R}^2}1_{|y|>1/n}\psi(y)\,\nu(\mathrm{d}x,\mathrm{d}y).
\end{align*}
Then $F$ is a square-integrable measurable function of $(X_t)_{t\geq 0}$ and $G$ is a square-integrable measurable function of $(Y_t)_{t\geq 0}$. Moreover, we have that $E[F]=E[G]=0$ and that
\[
\begin{aligned}
\Var(F)&=\int_{\mathbb{R}^2}(\varphi(x))^2\,\nu(\mathrm{d}x,\mathrm{d}y),\\ \Var(G)&=\int_{\mathbb{R}^2}(\psi(y))^2\,\nu(\mathrm{d}x,\mathrm{d}y),\\ \Cov(F,G)&=\int_{\mathbb{R}^2}\varphi(x)\psi(y)\,\nu(\mathrm{d}x,\mathrm{d}y).
\end{aligned}
\]
Hence, we have that
\begin{equation*}
R((X_t)_{t\geq 0},(Y_t)_{t\geq 0})\geq\rho(F,G)\geq\mathrm{Op}(\nu)-\varepsilon.
\end{equation*}
Since $\varepsilon>0$ is arbitrary, we see that
\begin{equation}\label{eq: R geq Op nu}
R((X_t)_{t\geq 0},(Y_t)_{t\geq 0})\geq\mathrm{Op}(\nu).
\end{equation}
Hence, by \eqref{eq: R geq |rho|} and \eqref{eq: R geq Op nu}, we have that $R((X_t)_{t\geq 0},(Y_t)_{t\geq 0})\geq\max(|\rho|,\mathrm{Op}(\nu))$.

To prove the reverse inequality $R((X_t)_{t\geq 0},(Y_t)_{t\geq 0})\leq\max(|\rho|,\mathrm{Op}(\nu))$, we fix $T>0$ and square-integrable non-degenerate $F=F((X_t)_{0\leq t\leq T})$ and $G=G((Y_t)_{0\leq t\leq T})$, where $F(\cdot)$ and $G(\cdot)$ are measurable functions. By \cite[Theorem~12.5]{BillingsleyMR1700749}, the Borel $\sigma$-field $\mathcal{D}$ of the Skorokhod space $D[0,T]$ is equal to the sigma-field generated by finite-dimensional cylinders. Hence, we have that $F\in\sigma(X_t,0\leq t\leq T)\subset\mathcal{F}^{X}_{T}$ and $G\in\sigma(Y_t,0\leq t\leq T)\subset\mathcal{F}^{Y}_T$. By Theorem~\ref{thm: martingale representation}, we find square-integrable predictable processes $\varphi^{X}$, $\psi^{X}$, $\varphi^{Y}$ and $\psi^{Y}$ such that $F=E[F]+F_{B}+F_{N}$ and $G=E[G]+G_{B}+G_{N}$ with probability $1$, where
\begin{equation}
\begin{aligned}
&F_{B}=\int_{0}^{T}\varphi^{X}(t)\,\mathrm{d}B^{X}_t,\\
&G_{B}=\int_{0}^{T}\varphi^{Y}(t)\,\mathrm{d}B^{Y}_t,\\
&F_{N}=\int_{0}^{T}\int_{\mathbb{R}\setminus\{0\}}\psi^{X}(t,x)\widetilde{N}^{X}(\mathrm{d}t,\mathrm{d}x) =\int_{0}^{T}\int_{\mathbb{R}\setminus\{0\}}\int_{\mathbb{R}}\psi^{X}(t,x)\widetilde{N}(\mathrm{d}t,\mathrm{d}x,\mathrm{d}y),\\
&G_{N}=\int_{0}^{T}\int_{\mathbb{R}\setminus\{0\}}\psi^{Y}(t,y)\widetilde{N}^{Y}(\mathrm{d}t,\mathrm{d}y) =\int_{0}^{T}\int_{\mathbb{R}}\int_{\mathbb{R}\setminus\{0\}}\psi^{Y}(t,y)\widetilde{N}(\mathrm{d}t,\mathrm{d}x,\mathrm{d}y).
\end{aligned}
\end{equation}
\begin{claim}\label{claim: orthogonality}
We have that $E[F_{B}F_{N}]=E[F_{B}G_{N}]=E[G_{B}F_{N}]=E[G_{B}G_{N}]=0$.
\end{claim}
\begin{proof}
We only provide the proof for $E(F_{B}G_{N})=0$, as the other equations can be proved similarly. We define a martingale-valued measure $M$ as follows: for $t\geq 0$ and $A\subset\mathbb{R}^2$, we define
\[M(t,A)=\int_{0}^{t}\int_{A\setminus\{(0,0)\}}\widetilde{N}(\mathrm{d}s,\mathrm{d}x,\mathrm{d}y) +B^{X}_t\delta_{(0,0)}(A).\]
Let $\mu(t,A)=E(M(t,A)^2)$. By independence between $\widetilde{N}$ and $B^{X}$, we get that
\[\mu(t,A)=t\nu(A\setminus\{0\})+t\Sigma_{11}\delta_{(0,0)}(A).\]
Let $f(t,x,y)=1_{x=y=0}\varphi^{X}(t)$ and $g(t,x,y)=1_{y\neq 0}\psi^{Y}(t,y)$. Then $F_{B}$ is equal to the stochastic integration $\int_{0}^{T}f(t,x,y)M(\mathrm{d}t,\mathrm{d}x,\mathrm{d}y)$ and $G_{N}$ is equal to the stochastic integration $\int_{0}^{T}g(t,x,y)M(\mathrm{d}t,\mathrm{d}x,\mathrm{d}y)$. For the general theory of stochastic integration against a certain type of martingale-valued measure, we refer to \cite[Chapter~4]{ApplebaumMR2512800}. By It\^{o}'s isometry for stochastic integrals (see \cite[Theorem~4.2.3 and Exercise~4.2.4]{ApplebaumMR2512800}), we have that
\[E[F_{B}G_{N}]=\int_{0}^{T}\int_{\mathbb{R}^2} f(t,x,y)g(t,x,y)\,\mu(\mathrm{d}t,\mathrm{d}x,\mathrm{d}y)=0\]
by the definitions of $f$ and $g$.
\end{proof}
By Claim~\ref{claim: orthogonality}, we have that
\begin{equation}\label{eq: Var Cov sep B N}
\left\{
\begin{aligned}
\Var(F)&=E[F_{B}^2]+E[F_{N}^2],\\
\Var(G)&=E[G_{B}^2]+E[G_{N}^2],\\
\Cov(F,G)&=E[F_{B}G_{B}]+E[F_{N}G_{N}].
\end{aligned}
\right.
\end{equation}
By It\^{o}'s isometry for stochastic integrals, we have that
\begin{equation*}
\begin{aligned}
E[F_{B}^2]&=\Sigma_{11}\int_{0}^{T}(\varphi^{X}(t))^2\,\mathrm{d}t,\\ E[G_{B}^2]&=\Sigma_{22}\int_{0}^{T}(\varphi^{Y}(t))^2\,\mathrm{d}t,\\ E[F_{B}G_{B}]&=\Sigma_{12}\int_{0}^{T}\varphi^{X}(t)\varphi^{Y}(t)\,\mathrm{d}t.
\end{aligned}
\end{equation*}
By the Cauchy-Schwarz inequality, we obtain that
\begin{equation}\label{eq: EFBGB|rho|sqrtEFB2EGB2}
E[F_{B}G_{B}]\leq |\rho|\sqrt{E(F_{B}^2)E(G_{B}^2)}.
\end{equation}
Next, we wish to prove that
\begin{equation}\label{eq: EFNGNOpvsqrtEFN2EGN2}
E[F_{N}G_{N}]\leq \mathrm{Op}(\nu)\sqrt{E(F_{N}^2)E(G_{N}^2)}.
\end{equation}
Similarly, by It\^{o}'s isometry for stochastic integrals against compensated Poisson random measure, we get that
\begin{equation*}
\begin{aligned}
E[F_{N}^2]&=\int_{0}^{T}\int_{\mathbb{R}^2}1_{x\neq 0}(\psi^{X}(t,x))^2\,\mathrm{d}t\nu(\mathrm{d}x,\mathrm{d}y),\\ E[G_{N}^2]&=\int_{0}^{T}\int_{\mathbb{R}^2}1_{y\neq 0}(\psi^{Y}(t,y))^2\,\mathrm{d}t\nu(\mathrm{d}x,\mathrm{d}y),\\ E[F_{N}G_{N}]&=\int_{0}^{T}\int_{\mathbb{R}^2}1_{x\neq 0,y\neq 0}\psi^{X}(t,x)\psi^{Y}(t,y)\,\mathrm{d}t\nu(\mathrm{d}x,\mathrm{d}y).
\end{aligned}
\end{equation*}
By the definition of $\mathrm{Op}(\nu)$, we have that
\begin{multline}\label{eq: defn Op applied}
\int_{\mathbb{R}^2}1_{x\neq 0,y\neq 0}\psi^{X}(t,x)\psi^{Y}(t,y)\,\nu(\mathrm{d}x,\mathrm{d}y)\\
\leq\mathrm{Op}(\nu)\sqrt{\int_{\mathbb{R}^2}1_{x\neq 0}(\psi^{X}(t,x))^2\,\nu(\mathrm{d}x,\mathrm{d}y)}\sqrt{\int_{\mathbb{R}^2}1_{y\neq 0}(\psi^{Y}(t,y))^2\,\nu(\mathrm{d}x,\mathrm{d}y)}.
\end{multline}
Integrating both sides of \eqref{eq: defn Op applied} against $\mathrm{d}t$ and using the Cauchy-Schwarz inequality, we prove \eqref{eq: EFNGNOpvsqrtEFN2EGN2}. Finally, by \eqref{eq: Var Cov sep B N}, \eqref{eq: EFBGB|rho|sqrtEFB2EGB2}, \eqref{eq: EFNGNOpvsqrtEFN2EGN2} and the Cauchy-Schwarz inequality, we have that
\begin{align*}
\Cov(F,G)&\leq\max(|\rho|,\mathrm{Op}(\nu))\left(\sqrt{E[F_{B}^2]E[G_{B}^2]}+\sqrt{E[F_{N}^2]E[G_{N}^2]}\right)\\
&\leq \max(|\rho|,\mathrm{Op}(\nu))\sqrt{E[F_{B}^2]+E[F_{N}^2]}\sqrt{E[G_{B}^2]+E[G_{N}^2]}\\
&=\max(|\rho|,\mathrm{Op}(\nu))\sqrt{\Var(F)}\sqrt{\Var(G)}.
\end{align*}
Since the above inequality holds for all proper $F=F((X_t)_{0\leq t\leq T})$ and $G=G((Y_t)_{0\leq t\leq T})$, we get that
\[R((X_t)_{0\leq t\leq T},(Y_t)_{0\leq t\leq T})\leq \max(|\rho|,\mathrm{Op}(\nu)).\]
As $T\to\infty$, $((X_t)_{0\leq t\leq T},(Y_t)_{0\leq t\leq T})$ converges to $((X_t)_{t\geq 0},(Y_t)_{t\geq 0})$. By Lemma~\ref{lem: lower semi-continuity} (lower semi-continuity), we have that
\[R((X_t)_{t\geq 0},(Y_t)_{t\geq 0})\leq \liminf_{T\to\infty}R((X_t)_{0\leq t\leq T},(Y_t)_{0\leq t\leq T})\leq\max(|\rho|,\mathrm{Op}(\nu)),\]
and the proof is complete.

\subsection{Proof of Theorem~\ref{thm: stable}}
In this subsection, we will prove Theorem~\ref{thm: stable}. Firstly, we briefly explain the ideas. We have found the maximal correlation coefficient for a two-dimensional L\'{e}vy process in Theorem~\ref{thm: Levy}. Then we apply Theorem~\ref{thm: Levy} to a two-dimensional $\alpha$-stable process $(X_t,Y_t)_{t\geq 0}$ with $0<\alpha<2$. The Brownian part vanishes. Hence, it suffices to calculate $\mathrm{Op}(\nu)$. For a stable process, $\nu$ takes the special form \eqref{eq: stable Levy measure spectral measure}. If $\tau(\mathrm{d}\theta)$ has a density $\tau(\theta)$ with respect to the Lebesgue measure $\mathrm{d}\theta$, the L\'{e}vy measure $\nu(\mathrm{d}x,\mathrm{d}y)$ also has a density $\nu(x,y)$. Moreover, if $\tau(-\theta)=\tau(\theta)$, we have $\nu(-x,-y)=\nu(x,y)$. In this case, we can define a homogeneous kernel
\[K(x,y)=\frac{\nu(x,y)}{\sqrt{\nu_{X}(x)\nu_{Y}(y)}}\]
of degree $-1$, where
\[\nu_{X}(x)=\int_{\mathbb{R}}\nu(x,y)\,\mathrm{d}y,\nu_{Y}(y)=\int_{\mathbb{R}}\nu(x,y)\,\mathrm{d}x,\text{ and }K(\lambda x,\lambda y)=|\lambda|^{-1}K(x,y)\]
for all $x,y\in\mathbb{R}$ and $\lambda\neq 0$. The kernel $K(x,y)$ induces a linear operator $K$ on the $L^2$-space by
\[K\psi(x):=\int_{\mathbb{R}}K(x,y)\psi(y)\,\mathrm{d}y.\]
Then the constant $\mathrm{Op}(\nu)$ is just the operator norm
\[\|K\|=\sup_{\|\psi\|_2>0}\frac{\|K\psi\|_2}{\|\psi\|_2}.\]
Here, ``$\mathrm{Op}$'' is short for the word ``operator''. The determination of the norm $\|K\|$ of the homogeneous kernel $K$ of degree $-1$ is a classical problem under the name ``Hilbert-Hardy inequality'', see \cite[Theorem~319]{HardyLittlewoodPolyaMR944909} and \cite[Theorem~42.9]{YangMR2962669}. However, in our case, $\tau(\mathrm{d}\theta)$ is not necessarily absolutely continuous with respect to the Lebesgue measure. Moreover, even if it is possible to define the kernel $K(x,y)$ in certain cases, because of the absence of the symmetry $\tau(-\theta)=\tau(\theta)$, $K(x,y)$ is not homogeneous. Instead, $K(x,y)$ is only positively homogeneous, that is, $K(\lambda x,\lambda y)=\lambda^{-1}K(x,y)$ for $x,y\in\mathbb{R}$ and $\lambda>0$. The difference between homogeneity and positive homogeneity results in different expressions for $\|K\|$. Indeed, in the case where $K$ is positive homogeneous, $\|K\|=\mathrm{Op}(\nu)$ is expressed as the spectral norm of a $2\times 2$ matrix, see Lemma~\ref{lem: stable Op nu}. However, in the case where $K(x,y)$ is a homogeneous kernel, by \cite[Theorem~42.9]{YangMR2962669}, we have that
\[\|K\|=\int_{-\infty}^{\infty}K(x,1)|x|^{-\frac{1}{2}}\,\mathrm{d}x.\]
Equivalently, in terms of $\alpha$ and $\tau$, we have that
\[\|K\|=\frac{\int_{0}^{2\pi}|\cos\theta\sin\theta|^{\frac{\alpha}{2}}\tau(\theta)\,\mathrm{d}\theta}{\sqrt{\int_{0}^{2\pi}|\cos\theta|^{\alpha}\tau(\theta)\,\mathrm{d}\theta\int_{0}^{2\pi}|\sin\theta|^{\alpha}\tau(\theta)\,\mathrm{d}\theta}}.\]
We are not aware of existing results for $\mathrm{Op}(\nu)$ in general. Hence, we calculate $\mathrm{Op}(\nu)$ in Lemma~\ref{lem: stable Op nu} by an adaptation of the argument leading to the Hilbert-Hardy inequalities. In this way, we subsequently find the maximal correlation coefficient $R((X_t)_{t\geq 0},(Y_t)_{t\geq 0})$ for a two-dimensional stable process. As $\alpha$-stable distributions are marginal distributions of $\alpha$-stable processes, we have the upper bound $R(X_1,Y_1)\leq R((X_t)_{t\geq 0},(Y_t)_{t\geq 0})$. The reverse inequality is a consequence of the convergence of stable random walks towards stable processes, Theorem~\ref{thm: random walks} and Lemma~\ref{lem: lower semi-continuity}. However, we are unable to generalize Theorem~\ref{thm: stable} to general infinitely divisible distributions. Indeed, it is possible to have $R(X_1,Y_1)<R((X_t)_{t\geq 0},(Y_t)_{t\geq 0})$ for a general two-dimensional L\'{e}vy process, see Examples~\ref{exam: 3} and \ref{exam: 4}.

\begin{lemma}\label{lem: stable Op nu}
Suppose that the L\'{e}vy measure has the form \eqref{eq: stable Levy measure spectral measure}. Let
\begin{equation*}
\begin{aligned}
&C_{++}=\int_{0}^{\frac{\pi}{2}}|\cos\theta\sin\theta|^{\frac{\alpha}{2}}\,\tau(\mathrm{d}\theta),\quad &C_{+-}=\int_{\frac{3\pi}{2}}^{2\pi}|\cos\theta\sin\theta|^{\frac{\alpha}{2}}\,\tau(\mathrm{d}\theta),\\
&C_{-+}=\int_{\frac{\pi}{2}}^{\pi}|\cos\theta\sin\theta|^{\frac{\alpha}{2}}\,\tau(\mathrm{d}\theta),\quad &C_{--}=\int_{\pi}^{\frac{3\pi}{2}}|\cos\theta\sin\theta|^{\frac{\alpha}{2}}\,\tau(\mathrm{d}\theta),
\end{aligned}
\end{equation*}
and
\begin{equation*}
\begin{aligned}
&D^{x}_{+}=\int_{-\frac{\pi}{2}}^{\frac{\pi}{2}}|\cos\theta|^{\alpha}\,\tau(\mathrm{d}\theta),\quad &D^{x}_{-}=\int_{\frac{\pi}{2}}^{\frac{3\pi}{2}}|\cos\theta|^{\alpha}\,\tau(\mathrm{d}\theta),\\
&D^{y}_{+}=\int_{0}^{\pi}|\sin\theta|^{\alpha}\,\tau(\mathrm{d}\theta),\quad &D^{y}_{-}=\int_{\pi}^{2\pi}|\sin\theta|^{\alpha}\,\tau(\mathrm{d}\theta).
\end{aligned}
\end{equation*}
Then we have that
\[\mathrm{Op}(\nu)=\left\|\begin{pmatrix}
C_{++}/\sqrt{D^{x}_{+}D^{y}_{+}} & \quad C_{+-}/\sqrt{D^{x}_{+}D^{y}_{-}}\\
C_{-+}/\sqrt{D^{x}_{-}D^{y}_{+}} & \quad C_{--}/\sqrt{D^{x}_{-}D^{y}_{-}}
\end{pmatrix}\right\|_2\]
with the convention that $0/0=0$, where $\|\cdot\|_2$ denotes the spectral norm.
\end{lemma}
\begin{proof}
Using the polar coordinates, by \eqref{eq: stable Levy measure spectral measure}, we have that
\[L:=\int_{\mathbb{R}^2}\varphi(x)\psi(y)\,\nu(\mathrm{d}x,\mathrm{d}y)=\int_{0}^{\infty}\int_{\mathbb{R}/(2\pi\mathbb{Z})} \varphi(r\cos\theta)\psi(r\sin\theta)\frac{1}{r^{1+\alpha}}\,\mathrm{d}r\tau(\mathrm{d}\theta).\]
Since we require that $\varphi(0)=\psi(0)=0$, we may assume that $\cos\theta\neq 0$ and $\sin\theta\neq 0$ in the above integral. Write
\[L=L_{++}+L_{-+}+L_{--}+L_{+-},\]
where
\begin{equation*}
\begin{aligned}
L_{++}=\int_{0}^{\infty}\int_{0}^{\frac{\pi}{2}} \varphi(r\cos\theta)\psi(r\sin\theta)\frac{1}{r^{1+\alpha}}\,\mathrm{d}r\tau(\mathrm{d}\theta),\\
L_{-+}=\int_{0}^{\infty}\int_{\frac{\pi}{2}}^{\pi} \varphi(r\cos\theta)\psi(r\sin\theta)\frac{1}{r^{1+\alpha}}\,\mathrm{d}r\tau(\mathrm{d}\theta),\\
L_{--}=\int_{0}^{\infty}\int_{\pi}^{\frac{3\pi}{2}} \varphi(r\cos\theta)\psi(r\sin\theta)\frac{1}{r^{1+\alpha}}\,\mathrm{d}r\tau(\mathrm{d}\theta),\\
L_{+-}=\int_{0}^{\infty}\int_{\frac{3\pi}{2}}^{2\pi} \varphi(r\cos\theta)\psi(r\sin\theta)\frac{1}{r^{1+\alpha}}\,\mathrm{d}r\tau(\mathrm{d}\theta).\\
\end{aligned}
\end{equation*}
By writing the integrand as the product of \[\varphi(r\cos\theta)r^{-\frac{1+\alpha}{2}}|\tan\theta|^{\frac{\alpha}{4}}\text{ and }\psi(r\sin\theta)r^{-\frac{1+\alpha}{2}}|\cot\theta|^{\frac{\alpha}{4}},\]
using the Cauchy-Schwarz inequality, we obtain that
\[L_{++}\leq\sqrt{I^{\varphi}_{++}I^{\psi}_{++}},\]
where
\[I^{\varphi}_{++}=\int_{0}^{\infty}\int_{0}^{\frac{\pi}{2}}(\varphi(r\cos\theta))^2\frac{1}{r^{1+\alpha}}|\tan\theta|^{\frac{\alpha}{2}}\,\mathrm{d}r\tau(\mathrm{d}\theta)\]
and
\[I^{\psi}_{++}=\int_{0}^{\infty}\int_{0}^{\frac{\pi}{2}}(\psi(r\sin\theta))^2\frac{1}{r^{1+\alpha}}|\cot\theta|^{\frac{\alpha}{2}}\,\mathrm{d}r\tau(\mathrm{d}\theta).\]
By performing the change of variable $x=r\cos\theta$, we get that
\begin{equation*}
I^{\varphi}_{++}=\int_{0}^{\frac{\pi}{2}}\int_{0}^{\infty}(\varphi(x))^2\frac{1}{|x|^{1+\alpha}}|\cos\theta\sin\theta|^{\frac{\alpha}{2}}\,\tau(\mathrm{d}\theta)\mathrm{d}x=C_{++}F_{+},
\end{equation*}
where
\begin{equation*}
F_{+}=\int_{0}^{\infty}(\varphi(x))^2\frac{1}{|x|^{1+\alpha}}\,\mathrm{d}x.
\end{equation*}
Similarly, by performing the change of variable $y=r\sin\theta$, we get that
\begin{equation*}
I^{\psi}_{++}=C_{++}G_{+},
\end{equation*}
where
\begin{equation*}
G_{+}=\int_{0}^{\infty}(\psi(y))^2\frac{1}{|y|^{1+\alpha}}\,\mathrm{d}y.
\end{equation*}
Hence, we have that
\begin{equation}
L_{++}\leq C_{++}\sqrt{F_{+}}\sqrt{G_{+}}.
\end{equation}
Similarly, we have that
\begin{equation*}
L_{+-}\leq C_{+-}\sqrt{F_{+}}\sqrt{G_{-}},\quad L_{-+}\leq C_{-+}\sqrt{F_{-}}\sqrt{G_{+}},\quad L_{--}\leq C_{--}\sqrt{F_{-}}\sqrt{G_{-}},
\end{equation*}
where
\begin{equation*}
F_{-}=\int_{-\infty}^{0}(\varphi(x))^2\frac{1}{|x|^{1+\alpha}}\,\mathrm{d}x,\quad G_{-}=\int_{-\infty}^{0}(\psi(y))^2\frac{1}{|y|^{1+\alpha}}\,\mathrm{d}y.
\end{equation*}
Hence, we have that
\begin{equation}\label{eq: varphi x psi y leq CFG}
\begin{aligned}
\int_{\mathbb{R}^2}\varphi(x)\psi(y)\,\nu(\mathrm{d}x,\mathrm{d}y)\leq & C_{++}\sqrt{F_{+}}\sqrt{G_{+}}+C_{+-}\sqrt{F_{+}}\sqrt{G_{-}}\\
&+C_{-+}\sqrt{F_{-}}\sqrt{G_{+}}+C_{--}\sqrt{F_{-}}\sqrt{G_{-}}.
\end{aligned}
\end{equation}
Similarly, we find that
\begin{equation}\label{eq: varphi x 2 leq DF}
\int_{\mathbb{R}^2}(\varphi(x))^2\,\nu(\mathrm{d}x,\mathrm{d}y)=D^{x}_{+}F_{+}+D^{x}_{-}F_{-}
\end{equation}
and that
\begin{equation}\label{eq: psi y 2 leq DG}
\int_{\mathbb{R}^2}(\psi(y))^2\,\nu(\mathrm{d}x,\mathrm{d}y)=D^{y}_{+}G_{+}+D^{y}_{-}G_{-}.
\end{equation}
Let
\[A=\begin{pmatrix}
C_{++}/\sqrt{D^{x}_{+}D^{y}_{+}} & \quad C_{+-}/\sqrt{D^{x}_{+}D^{y}_{-}}\\
C_{-+}/\sqrt{D^{x}_{-}D^{y}_{+}} & \quad C_{--}/\sqrt{D^{x}_{-}D^{y}_{-}}
\end{pmatrix}\]
with the convention that $0/0=0$. Let $\|A\|_2$ be the spectral norm of $A$. Then by \eqref{eq: varphi x psi y leq CFG}, \eqref{eq: varphi x 2 leq DF} and \eqref{eq: psi y 2 leq DG}, we have that
\[\int_{\mathbb{R}^2}\varphi(x)\psi(y)\,\nu(\mathrm{d}x,\mathrm{d}y)\leq \|A\|_2\sqrt{\int_{\mathbb{R}^2}(\varphi(x))^2\,\nu(\mathrm{d}x,\mathrm{d}y)\int_{\mathbb{R}^2}(\psi(y))^2\,\nu(\mathrm{d}x,\mathrm{d}y)}.\]
Hence, $\mathrm{Op}(\nu)\leq \|A\|_2$. It remains to prove the reverse inequality $\mathrm{Op}(\nu)\geq \|A\|_2$. Without loss of generality, we may assume that $\|A\|_2>0$. Equivalently, we assume that $D^{x}_{+}+D^{x}_{-}>0$ and $D^{y}_{+}+D^{y}_{-}>0$. For this purpose, for $\varepsilon>0$, we set
\[\varphi(x)=1_{x>1}|x|^{\frac{\alpha}{2}-\varepsilon}b^{\varphi}_{+}+1_{x<-1}|x|^{\frac{\alpha}{2}-\varepsilon}b^{\varphi}_{-},\quad \psi(y)=1_{y>1}|y|^{\frac{\alpha}{2}-\varepsilon}b^{\psi}_{+}+1_{y<-1}|y|^{\frac{\alpha}{2}-\varepsilon}b^{\psi}_{-},\]
where the positive constants $b^{\varphi}_{+}$, $b^{\varphi}_{-}$, $b^{\psi}_{+}$ and $b^{\psi}_{-}$ will be chosen later. Then we have that
\begin{align*}
L_{++}&=b^{\varphi}_{+}b^{\psi}_{+}\int_{0}^{\infty}\int_{0}^{\frac{\pi}{2}}1_{r\cos\theta>1,r\sin\theta>1} |r\cos\theta|^{\frac{\alpha}{2}-\varepsilon}|r\sin\theta|^{\frac{\alpha}{2}-\varepsilon}\frac{1}{r^{1+\alpha}}\,\mathrm{d}r\tau(\mathrm{d}\theta)\\
&=b^{\varphi}_{+}b^{\psi}_{+}\int_{0}^{\infty}\int_{0}^{\frac{\pi}{2}}1_{r\cos\theta>1,r\sin\theta>1} |\cos\theta\sin\theta|^{\frac{\alpha}{2}-\varepsilon}\frac{1}{r^{1+2\varepsilon}}\,\mathrm{d}r\tau(\mathrm{d}\theta)\\
&=\frac{1}{2\varepsilon}b^{\varphi}_{+}b^{\psi}_{+}\int_{0}^{\frac{\pi}{2}} |\cos\theta\sin\theta|^{\frac{\alpha}{2}-\varepsilon}(\min(|\cos\theta|,|\sin\theta|))^{2\varepsilon}\,\tau(\mathrm{d}\theta).
\end{align*}
Similarly, we have that
\begin{equation*}
L_{+-}=\frac{1}{2\varepsilon}b^{\varphi}_{+}b^{\psi}_{-}\int_{\frac{3\pi}{2}}^{2\pi} |\cos\theta\sin\theta|^{\frac{\alpha}{2}-\varepsilon}(\min(|\cos\theta|,|\sin\theta|))^{2\varepsilon}\,\tau(\mathrm{d}\theta),
\end{equation*}
\begin{equation*}
L_{-+}=\frac{1}{2\varepsilon}b^{\varphi}_{-}b^{\psi}_{+}\int_{\frac{\pi}{2}}^{\pi} |\cos\theta\sin\theta|^{\frac{\alpha}{2}-\varepsilon}(\min(|\cos\theta|,|\sin\theta|))^{2\varepsilon}\,\tau(\mathrm{d}\theta),
\end{equation*}
and
\begin{equation*}
L_{--}=\frac{1}{2\varepsilon}b^{\varphi}_{-}b^{\psi}_{-}\int_{\pi}^{\frac{3\pi}{2}} |\cos\theta\sin\theta|^{\frac{\alpha}{2}-\varepsilon}(\min(|\cos\theta|,|\sin\theta|))^{2\varepsilon}\,\tau(\mathrm{d}\theta).
\end{equation*}
Moreover, we find that
\begin{equation*}
\int_{\mathbb{R}^2}(\varphi(x))^2\,\nu(\mathrm{d}x,\mathrm{d}y) =D^{x}_{+}F_{+}+D^{x}_{-}F_{-} =\frac{1}{2\varepsilon}D^{x}_{+}(b^{\varphi}_{+})^2+\frac{1}{2\varepsilon}D^{x}_{-}(b^{\varphi}_{-})^2.
\end{equation*}
and
\begin{equation*}
\int_{\mathbb{R}^2}(\psi(y))^2\,\nu(\mathrm{d}x,\mathrm{d}y)   =\frac{1}{2\varepsilon}D^{y}_{+}(b^{\psi}_{+})^2+\frac{1}{2\varepsilon}D^{y}_{-}(b^{\psi}_{-})^2.
\end{equation*}
We chose $b^{\varphi}_{+},b^{\varphi}_{-},b^{\psi}_{+}$ and $b^{\psi}_{-}$ such that
\begin{multline*}
(\sqrt{D^{x}_{+}}b^{\varphi}_{+},\sqrt{D^{x}_{-}}b^{\varphi}_{-})A\begin{pmatrix}
\sqrt{D^{y}_{+}}b^{\psi}_{+}\\
\sqrt{D^{y}_{-}}b^{\psi}_{-}
\end{pmatrix}\\
=\|A\|_2\sqrt{(D^{x}_{+}(b^{\varphi}_{+})^2+D^{x}_{-}(b^{\varphi}_{-})^2) (D^{y}_{+}(b^{\psi}_{+})^2+D^{y}_{-}(b^{\psi}_{-})^2)}.
\end{multline*}
Then we get that
\[\mathrm{Op}(\nu)\geq\lim_{\varepsilon\to 0}\frac{\int_{\mathbb{R}^2}\varphi(x)\psi(y)\,\nu(\mathrm{d}x,\mathrm{d}y)}{\sqrt{\int_{\mathbb{R}^2}(\varphi(x))^2\,\nu(\mathrm{d}x,\mathrm{d}y)\int_{\mathbb{R}^2}(\psi(y))^2\,\nu(\mathrm{d}x,\mathrm{d}y)}} =\|A\|_2.\]
\end{proof}
Currently, we have found an expression for $R((X_t)_{t\geq 0},(Y_t)_{t\geq 0})$. Clearly, we have that
\[R(X_1,Y_1)\leq R((X_t)_{t\geq 0},(Y_t)_{t\geq 0}).\]
On the other hand, consider the random walk $(S_n,T_n)_{n\geq 0}$ such that its increment $(S_{n+1}-S_{n},T_{n+1}-T_{n})$ has the same distribution as $(X_1,Y_1)$. By Theorem~\ref{thm: random walks}, we have that
\[R((S_n)_{n\geq 0},(T_n)_{n\geq 0})=R(X_1,Y_1).\]
Since $(X_1,Y_1)$ is stable, there exists $(c_n,d_n)\in\mathbb{R}$ such that $(S_n-c_n,T_n-d_n)$ has the same distribution as $n^{1/\alpha}(X_1,Y_1)$. Define
\[(X^{(n)}_t,Y^{(n)}_t)=\left(\frac{S_{[nt]}-c_{[nt]}}{n^{1/\alpha}},\frac{T_{[nt]}-d_{[nt]}}{n^{1/\alpha}}\right).\]
Then $R(X^{(n)},Y^{(n)})=R(S,T)=R(X_1,Y_1)$. Moreover, by \cite[Theorem~23.14]{KallenbergMR4226142}, the stochastic process $(X^{(n)}_t,Y^{(n)}_t)_{t\geq 0}$ converges towards $(X_t,Y_t)_{t\geq 0}$ in the Skorokhod space $D_{\mathbb{R}^2}[0,\infty)$. Hence, by Lemma~\ref{lem: lower semi-continuity}, we have that
\[R((X_t)_{t\geq 0},(Y_t)_{t\geq 0})\leq\liminf_{n\to\infty}R((X^{(n)}_t)_{t\geq 0},(Y^{(n)}_t)_{t\geq 0})=R(X_1,Y_1).\]
Finally, we find that $R(X_1,Y_1)=R((X_t)_{t\geq 0},(Y_t)_{t\geq 0})=\mathrm{Op}(\nu)$.

From the above argument, we have the following observation.
\begin{remark}\label{rem: doc lower bound}
Suppose that $(X,Y)$ is a stable random vector. Let $(\widetilde{X},\widetilde{Y})$ be in the domain of attraction of $(X,Y)$. Then we have that
\[R(\widetilde{X},\widetilde{Y})\geq R(X,Y),\]
where $R(X,Y)$ is given in Theorem~\ref{thm: stable}.
\end{remark}

\subsection{Examples}
Firstly, we consider an example studied in \cite{BrycDemboKaganMR2141340}. They proved the following theorem:
\begin{theorem}[Bryc-Dembo-Kagan]
Let $X$ and $Z$ be independent copies of $\alpha$-stable random variables with $0<\alpha\leq 2$. Then for all $\lambda\geq 0$, we have
\begin{equation}\label{eq: R X X lambda Z}
R(X,X+\lambda Z)=\frac{1}{\sqrt{1+|\lambda|^{\alpha}}}.
\end{equation}
If $X$ and $Z$ are symmetric, then the above inequality \eqref{eq: R X X lambda Z} holds for $\lambda<0$.
\end{theorem}
For $\alpha=2$, $X$ and $Z$ must be Gaussian, and the result goes back to Lancaster \cite{Lancaster}. The main contribution is in the case that $0<\alpha<2$, $\lambda\geq 0$, and $X$ are $\alpha$-stable. The result for $\lambda<0$ and symmetric $\alpha$-stable random variables can easily be deduced from \eqref{eq: R X X lambda Z} by taking $-Z$ instead of $Z$. The restriction to positive $\lambda$ in \eqref{eq: R X X lambda Z} has a reason. In general, the expression for $\lambda<0$ is different. Indeed, we have the following result:
\begin{proposition}
Let $X$ and $Z$ be independent copies of $\alpha$-stable random variables with $0<\alpha<2$. Then the L\'{e}vy measure $\nu_{X}(x)\mathrm{d}x$ of $X$ has the following form:
\[\nu_{X}(x)\mathrm{d}x=\frac{c_{-}}{|x|^{1+\alpha}}1_{x<0}\,\mathrm{d}x+\frac{c_{+}}{|x|^{1+\alpha}}1_{x>0}\,\mathrm{d}x.\]
For all $\lambda<0$, we have
\begin{equation}
R(X,X+\lambda Z)=1/\sqrt{1+\frac{\min(c_{-},c_{+})}{\max(c_{-},c_{+})}|\lambda|^{\alpha}}.
\end{equation}
\end{proposition}
The proof is based on Theorem~\ref{thm: stable} and the fact that $(X,X+\lambda Z)$ is stable.
\begin{proof}
By using characteristic functions and independence between $X$ and $Z$, we find that $(X,Y)=(X,X+\lambda Z)$ is an $\alpha$-stable random vector with L\'{e}vy measure $\nu(\mathrm{d}x,\mathrm{d}y)$, where
\begin{equation}
\nu(\mathrm{d}x,\mathrm{d}y)=\frac{1}{\lambda}\nu_{X}\left(\frac{y}{\lambda}\right)\delta_{0}(\mathrm{d}x)\mathrm{d}y+\nu_{X}(x)\mathrm{d}x\delta_{x}(\mathrm{d}y).
\end{equation}
If we write $\nu(\mathrm{d}x,\mathrm{d}y)$ in the form \eqref{eq: stable Levy measure spectral measure}, then we have
\[\tau(\mathrm{d}\theta)=c_{-}|\lambda|^{\alpha}\delta_{\frac{\pi}{2}}(\mathrm{d}\theta) +c_{+}|\lambda|^{\alpha}\delta_{\frac{3\pi}{2}}(\mathrm{d}\theta) +c_{+}\sqrt{2}^{\alpha}\delta_{\frac{\pi}{4}}(\mathrm{d}\theta) +c_{-}\sqrt{2}^{\alpha}\delta_{\frac{5\pi}{4}}(\mathrm{d}\theta).\]
Recall the notation in the statement of Theorem~\ref{thm: stable}. In our case, we have that
\[C_{++}=c_{+},C_{--}=c_{-},C_{+-}=C_{-+}=0,\]
\[D^{x}_{+}=c_{+},D^{x}_{-}=c_{-}, D^{y}_{+}=c_{+}+c_{-}|\lambda|^{\alpha}, D^{y}_{-}=c_{-}+c_{+}|\lambda|^{\alpha}.\]
Hence, by Theorem~\ref{thm: stable}, we have that
\begin{align*}
R(X,Y)&=\left\|\begin{pmatrix}
\sqrt{\frac{c_{+}}{c_{+}+c_{-}|\lambda|^{\alpha}}} &\quad 0\\
0 &\quad \sqrt{\frac{c_{-}}{c_{-}+c_{+}|\lambda|^{\alpha}}}
\end{pmatrix}\right\|_2\\
&=\max\left(\sqrt{\frac{c_{+}}{c_{+}+c_{-}|\lambda|^{\alpha}}},\sqrt{\frac{c_{-}}{c_{-}+c_{+}|\lambda|^{\alpha}}}\right)\\
&=1/\sqrt{1+\frac{\min(c_{-},c_{+})}{\max(c_{-},c_{+})}|\lambda|^{\alpha}}.
\end{align*}
\end{proof}

Secondly, we consider compound Poisson processes.
\begin{example}\label{exam: generic}
Consider $b=0$, $\Sigma=0$ and a probability measure $\nu(\mathrm{d}x,\mathrm{d}y)$ in \eqref{eq: Levy-Khinchine formula}. Suppose that
\[\nu(\{(x,y)\in\mathbb{R}^2\setminus\{(0,0)\}:x=0\text{ or }y=0\})=0.\]
Let $(X_t,Y_t)$ be the corresponding L\'{e}vy process. Then we have that
\[R((X_t)_{t\geq 0},(Y_t)_{t\geq 0})=\mathrm{Op}(\nu)=1.\]
The reason is that the number of jumps $N$ of $(X_t)_{t\geq 0}$ up to time $1$ is equal to that of $(Y_t)_{t\geq 0}$. Hence, $R((X_t)_{t\geq 0},(Y_t)_{t\geq 0})\geq R(N,N)=1$. So, for a generic compound Poisson process, the maximal correlation coefficient $R((X_t)_{t\geq 0},(Y_t)_{t\geq 0})$ is equal to $1$.
\end{example}

Finally, we show that the maximal correlation coefficient $R(X_1,Y_1)$ could be strictly less than $R((X_t)_{t\geq 0},(Y_t)_{t\geq 0})$.
\begin{example}\label{exam: 3}
Let $(B_t)_{t\geq 0}$ be a (one-dimensional) Brownian motion. Let $(N_t)_{t\geq 0}$ be a Poisson process of rate $1$ that is independent of $(B_t)_{t\geq 0}$. Let $X_t=B_t+N_t$. Then $(X_t,N_t)_{t\geq 0}$ is a two-dimensional L\'{e}vy process. By Theorem~\ref{thm: Levy}, we have that $R((X_t)_{t\geq 0},(N_t)_{t\geq 0})\geq \mathrm{Op}(\nu)=R((N_t)_{t\geq 0},(N_t)_{t\geq 0})=1$. Since the maximal correlation coefficient is bounded by $1$ from above, we have that $R((X_t)_{t\geq 0},(N_t)_{t\geq 0})=1$. However, $R(X_1,N_1)<1$. We prove this by contradiction. Suppose $R(X_1,N_1)=1$. Let \[f(x,n)\,\mathrm{d}x=P(X_1\in(x,x+\mathrm{d}x),N_1=n).\]
Then we have that
\[f(x,n)=\frac{e^{-1}}{n!}\frac{1}{\sqrt{2\pi}}e^{-(x-n)^2/2}.\]
Let $f_{X_1}(x)$ be the marginal density of $X_1$ and $f_{N_1}(n)=P(N_1=n)=\frac{e^{-1}}{n!}$. Define
\[k(x,n)=\frac{f(x,n)}{f_{X_1}(x)f_{N_1}(n)}\]
and
\begin{align*}
C^2&=\int_{-\infty}^{\infty}\sum_{n=0}^{\infty}(k(x,n)-1)^2f_{X_1}(x)f_{N_1}(n)\,\mathrm{d}x\\ &=\int_{-\infty}^{\infty}\sum_{n=0}^{\infty}\frac{f(x,n)^2}{f_{X_1}(x)f_{N_1}(n)}\,\mathrm{d}x-1\\ &=\int_{-\infty}^{\infty}\sum_{n=0}^{\infty}f_{N_1|X_1}(n|x)f_{X_1|N_1}(x|n)\,\mathrm{d}x-1,
\end{align*}
where $f_{N_1|X_1}(n|x)=P(N_1=n|X_1=x)$ and $f_{X_1|N_1}(x|n)$ is the conditional density of $X_1$ given $N_1=n$. Here, $C$ is called the \emph{mean square contingency} of $X_1$ and $N_1$. We will show that $C<\infty$ as follows: Note that
\begin{equation}\label{eq: f X1 N1 x n}
f_{X_1|N_1}(x|n)=\frac{1}{\sqrt{2\pi}}e^{-(x-n)^2/2}.
\end{equation}
Note that
\[f_{N_1|X_1}(n|x)=\frac{f(x,n)}{f_{X_1}(x)}=\frac{1}{f_{X_1}(x)}P(N_1=n)\frac{1}{\sqrt{2\pi}}e^{-(x-n)^2/2},\] \[f_{X_1}(x)=\sum_{n=0}^{\infty}f(x,n)=\sum_{n=0}^{\infty}P(N_1=n)\frac{1}{\sqrt{2\pi}}e^{-(x-n)^2/2}.\]
For $x\leq 0$, we have that
\begin{equation*}
f_{X_1}(x)\geq P(N_1=0)\frac{1}{\sqrt{2\pi}}e^{-x^2/2}=\frac{e^{-1}}{\sqrt{2\pi}}e^{-x^2/2}.
\end{equation*}
Hence, for $x\leq 0$, we have that
\begin{equation}\label{eq: f N1 X1 n x x<0}
f_{N_1|X_1}(n|x)\leq\frac{1}{n!}e^{-n^2/2-n|x|}.
\end{equation}
Hence, by \eqref{eq: f X1 N1 x n} and \eqref{eq: f N1 X1 n x x<0}, we see that
\[\int_{-\infty}^{0}\sum_{n=0}^{\infty}f_{N_1|X_1}(n|x)f_{X_1|N_1}(x|n)\,\mathrm{d}x<\infty.\]
It remains to prove that
\begin{equation}\label{eq: msc bound positive x}
\int_{0}^{\infty}\sum_{n=0}^{\infty}f_{N_1|X_1}(n|x)f_{X_1|N_1}(x|n)\,\mathrm{d}x<\infty.
\end{equation}
Hence, we assume that $x>0$ in the following. For fixed $x\geq 0$, let
\[g(x)=\max_{n\geq 0}f(x,n)=\max_{n\geq 0}P(N_1=n)\frac{1}{\sqrt{2\pi}}e^{-(x-n)^2/2}.\]
Then there exists $c>0$ such that for $x\geq e$, we have that
\[g(x)\geq f(x,\lceil x-\log x\rceil)\geq\frac{e^{-1}}{\sqrt{2\pi}}\frac{1}{\lceil x-\log x\rceil!}e^{-(\log x)^2/2}\geq \frac{c}{\lceil x\rceil^{10}}\frac{1}{\lceil x\rceil!}e^{(\log \lceil x\rceil)^2/2},\]
where $\lceil x\rceil$ is the least integer that is greater than or equal to $x$. Then
\[f_{X_1}(x)=\sum_{n=0}^{\infty}f(x,n)\geq\max_{n\geq 0}f(x,n)=g(x).\]
And for $x\geq e$, we have that
\[f_{N_1|X_1}(n|x)\leq\frac{1}{g(x)}\frac{e^{-1}}{n!}\frac{1}{\sqrt{2\pi}}e^{-(x-n)^2/2}\leq \frac{1}{c}(x+1)^{10}\frac{\lceil x\rceil!}{n!}e^{-(x-n)^2/2-(\log\lceil x\rceil)^2/2}.\]
Note that $\lceil x\rceil!/n!\leq 1$ for $n\geq \lceil x\rceil$ and $\lceil x\rceil!/n!\leq \lceil x\rceil^{\lceil x\rceil-n}$ for $0\leq n\leq\lceil x\rceil$. Hence, there exists $c>0$ and $C<\infty$ such that for $x>0$ and $|n-\lceil x\rceil|\leq \log\lceil x\rceil/10$,
\begin{equation}\label{eq: non-trivial bound for f N1 X1 n x}
f_{N_1|X_1}(n|x)\leq Ce^{-c(\log\lceil x\rceil)^2}.
\end{equation}
Besides, we have the trivial bound
\begin{equation}\label{eq: trivial bound for f N1 X1 n x}
f_{N_1|X_1}(n|x)=P(N_1=n|X_1=x)\leq 1\text{ for }|n-\lceil x\rceil|>\log\lceil x\rceil/10.
\end{equation}
Combining \eqref{eq: f X1 N1 x n}, \eqref{eq: non-trivial bound for f N1 X1 n x} and \eqref{eq: trivial bound for f N1 X1 n x}, we get \eqref{eq: msc bound positive x}. Hence, the mean-square contingency $C$ is finite. By \cite[Theorem~2]{RenyiMR0115203}, there exist non-degenerate $\varphi(X_1)$ and $\psi(N_1)$ such that $\rho(\varphi(X_1),\psi(N_1))=R(X_1,N_1)=1$. Therefore, there exist $c_1\neq 0,c_2\neq 0$ and $d\in\mathbb{R}$ such that $P(c_1\varphi(X_1)+c_2\psi(N_1)=d)=1$. However, given the value of $N_1$, $X_1$ could take any value in $\mathbb{R}$. If $P(c_1\varphi(X_1)+c_2\psi(N_1)=d)=1$, then $\varphi$ is constant almost everywhere, which contradicts with the non-degeneracy of $\varphi(X_1)$. Finally, by contradiction, we prove that $R(X_1,N_1)<1$.
\end{example}
For a compound Poisson process, it is still possible that $R(X_1,Y_1)<R((X_t)_{t\geq 0},(Y_t)_{t\geq 0})$, see the following example:
\begin{example}\label{exam: 4}
Let $(M_t)_{t\geq 0}$ and $(N_t)_{t\geq 0}$ be two independent Poisson processes with rate $1$. Let $X_t=M_t-N_t$ and $Y_t=M_t$. Then $(X_t,Y_t)_{t\geq 0}$ is a two-dimensional compound Poisson process. We have that $R((X_t)_{t\geq 0},(Y_t)_{t\geq 0})\geq \mathrm{Op}(\nu)=1$. Since the maximal correlation coefficient is at most one, we have that $R((X_t)_{t\geq 0},(Y_t)_{t\geq 0})=1$. However, we find that $R(X_1,Y_1)<1$ by numerical methods, although we do not have a theoretical proof at present. Indeed, let $X^{(n)}_1=\min(\max(X_1,-n),n)$ and $Y^{(n)}_1=\min(\max(Y_1,-n),n)$. Since $X^{(n)}_1$ is a measurable function of $X_1$ and $Y^{(n)}_1$ is a measurable function of $Y_1$, we have that $R(X^{(n)}_1,Y^{(n)}_1)\leq R(X_1,Y_1)$. On the other hand, since $(X^{(n)}_1,Y^{(n)}_1)$ converges to $(X_1,Y_1)$ in distribution, by Lemma~\ref{lem: lower semi-continuity}, we have that $R(X,Y)\leq\liminf_{n\to\infty}R(X^{(n)}_1,Y^{(n)}_1)$. Therefore, we must have
\[R(X,Y)=\lim_{n\to\infty}R(X^{(n)}_1,Y^{(n)}_1).\]
Since $(X^{(n)}_1,Y^{(n)}_1)$ takes values in a finite set, the maximal correlation $R(X^{(n)}_1,Y^{(n)}_1)$ is given by a certain eigenvalue of a finite matrix, which could be found by numerical methods. By numerical calculations, we find that $R(X,Y)=\lim_{n\to\infty}R(X^{(n)}_1,Y^{(n)}_1)$ is approximately $0.8321$.
\end{example}

\section{Analogs and generalizations of DKS inequality}\label{sect: MB like}

\subsection{Proof of the lower bound in Theorem~\ref{thm: R X (T,X_T)}}\label{subsect: proof of MB lower bound}
The lower bound
\[R((X_1,X_2,\ldots,X_n),(Y_1,Y_2,\ldots,Y_n))\geq\sqrt{\max\{P(i\in T):i\in[n]\}}\]
is a simple consequence of
\[R(X_i,Y_i)=\sqrt{P(i\in T)},\]
which follows from Lemma~\ref{lem: R(X,XB)} below.

Let $X$ be a non-degenerate random variable taking values in a general measurable space $(F,\mathcal{F})$. Let $B$ be a Bernoulli random variable independent of $X$. Assume that
\[P(B=1)=1-P(B=0)=p.\]
Suppose $\partial$ is a special point outside of $F$. Define
\[Y=\left\{\begin{array}{ll}
X, & \text{if }B=1,\\
\partial, & \text{if }B=0.
\end{array}\right.\]
\begin{lemma}\label{lem: R(X,XB)}
The maximal correlation coefficient between $X$ and $Y$ is equal to $\sqrt{p}$.
\end{lemma}
\begin{proof}
To calculate the maximal correlation coefficient between $X$ and $Y$, we take two functions $f$ and $g$ such that $E[f(X)]=0$, $\Var(f(X))<\infty$, $E[g(Y)]=0$ and $\Var(g(Y))<\infty$. Then we have that
\begin{equation}\label{eq: n=1 Var g(Y)}
\Var(g(Y))=E[(g(Y))^2]=E[E[(g(Y))^2|B]]=pE[(g(X))^2]+(1-p)(g(\partial))^2.
\end{equation}
We calculate the covariance between $f(X)$ and $g(Y)$ by taking the expectation conditionally on $B$:
\begin{align*}
\Cov(f(X),g(Y))&=E[f(X)g(Y)]\notag\\
&=E[E[f(X)g(Y)|B]]\notag\\
&=pE[f(X)g(X)]+(1-p)g(\partial)E[f(X)]\notag\\
&=pE[f(X)g(X)],
\end{align*}
where the first and the last inequalities are due to $E[f(X)]=0$. Combining previous results, we get that
\begin{align}\label{eq: n=1 rho f(X) g(Y)}
\rho(f(X),g(Y))&=\frac{\Cov(f(X),g(Y))}{\sqrt{\Var(f(X))}\sqrt{\Var(g(Y))}}\notag\\
&=\frac{pE[f(X)g(X)]}{\sqrt{E[(f(X))^2]}\sqrt{(1-p)(g(\partial))^2+pE[(g(X))^2]}}\\
&\leq\frac{pE[f(X)g(X)]}{\sqrt{E[(f(X))^2]}\sqrt{pE[(g(X))^2]}}\notag\\
&\leq\sqrt{p}\notag,
\end{align}
where we use the Cauchy-Schwarz inequality $(E[f(X)g(X)])^2\leq E[(f(X))^2]E[(g(X))^2]$ in the last step. Thus, we have proved that $R(X,Y)\leq\sqrt{p}$.

Finally, there exists $g$ such that
\[g(\partial)=0, E[g(X)|X\neq\partial]=0,\text{ and }0<E[(g(X))^2|X\neq\partial]<\infty.\]
Take $f=g$. Then
\[E[f(X)]=0\text{ and }0<\Var(f(X))=E[(g(X))^2]<\infty.\]
Moreover, $E[g(Y)]=pE[g(X)]=0$ and $0<\Var(g(Y))<\infty$ by \eqref{eq: n=1 Var g(Y)}. By \eqref{eq: n=1 rho f(X) g(Y)}, we have that
\[R(X,Y)\geq\rho(f(X),g(Y))=\sqrt{p},\]
and the proof is complete.
\end{proof}

\begin{remark}\label{rem: partial}
In the above argument, we need the condition that $\partial$ is outside of $F$ to ensure that the distribution of $X$ is non-degenerate conditionally on $X\neq\partial$, which guarantees the existence of $g$ such that $g(\partial)=0$, $E[g(X)|X\neq\partial]=0$ and $0<E[(g(X))^2|X\neq\partial]<\infty$. It is possible to allow $P(X=\partial)>0$. However, in such cases, we need to assume that the distribution of $X$ is non-degenerate conditionally on $X\neq\partial$. Otherwise, the formula is no longer correct. For example, we take $B$ to be a Bernoulli random variable with parameter $p=1/2$. Let $X$ be an independent copy of $B$. Take $\partial=0$. Then $Y=BX$. Note that $(X,Y)$ is a bivariate Bernoulli vector. By \eqref{eq: bivariate Bernoulli}, we find that $R(X,Y)=1/\sqrt{3}<\sqrt{p}=1/\sqrt{2}$.
\end{remark}

\subsection{Proof of Theorem~\ref{thm: R (S,X_S) (T,X_T)}}
In this subsection, we prove Theorem~\ref{thm: R (S,X_S) (T,X_T)}. The key ingredient is the analysis
of variance (ANOVA) decomposition developed in \cite[Appendix~I]{MadimanBarronMR2319376}. For the convenience of the readers, we briefly explain the ANOVA decomposition without proofs. Let $X_1,X_2,\ldots,X_n$ be independent random variables. Write $X=(X_1,X_2,\ldots,X_n)$. Suppose that $\psi:F^n\to\mathbb{R}$ belongs to $L^2$, that is, $\psi$ is a measurable function such that $E[\psi^2(X_1,X_2,\ldots,X_n)]<\infty$. For each $j\in[n]$, define $E_j\psi$ by
\[E_j\psi(x_1,x_2,\ldots,x_n)=E[\psi(X_1,X_2,\ldots,X_n)|X_i=x_i,\forall i\neq j].\]
In particular, $E_j\psi$ does not depend on $x_j$. For a subset $t\subset[n]$, define the linear subspace
\[\mathcal{H}_{t}=\{\psi\in L^2:E_j\psi=\psi 1_{j\notin t},\forall j\in[n]\}.\]
In particular, for $\psi\in\mathcal{H}_t$, $\psi$ does not depend on $x_j$ for $j\notin t$. Then $(\mathcal{H}_t)_{t\subset[n]}$ are orthogonal. Denote by $\overline{E}_t$ the orthogonal projection from $L^2$ onto $\mathcal{H}_t$. In fact,
\[\overline{E}_t=\prod_{j\in t}(I-E_j)\prod_{k\notin t}E_k,\]
where $I$ is the identity map. In particular, $\overline{E}_{\emptyset}$ is equal to the usual expectation $E$. Then we have the orthogonal decomposition
\[\psi=\sum_{t\subset[n]}\overline{E}_t\psi.\]
If $\psi$ depends solely on $(x_j)_{j\in s}$ for some $s\subset[n]$ and $t$ is not a subset of $s$, then $\overline{E}_t\psi=0$ by the definition of $\overline{E}_t$ and the fact that $E_j\psi=\psi$ for $j\in t\setminus s$. Hence, if $\psi$ depends solely on $(x_j)_{j\in s}$, then we have that
\[\psi=\sum_{t\subset s}\overline{E}_t\psi.\]
Sometimes, we encounter the random variable $\overline{E}_t\psi(X)$. Since $\overline{E}_t\psi$ depends solely on $(x_j)_{j\in t}$, we may write $\overline{E}_t\psi(X_{t})$ instead of $\overline{E}_t\psi(X)$, where $X_t$ means the subvector $(X_j)_{j\in t}$.

We now prove Theorem~\ref{thm: R (S,X_S) (T,X_T)}. For a vector $x=(x_1,x_2,\ldots,x_n)$ and a subset $s\subset[n]$, we denote by $x_s$ the subvector $(x_j)_{j\in s}$. In this way, we have $X_S=(X_j)_{j\in S}$ and $X_T=(X_j)_{j\in T}$. Since $\partial$ is outside of $(F,\mathcal{F})$, we have the following.
\[R((Y_1,Y_2,\ldots,Y_n),(Z_1,Z_2,\ldots,Z_n))=R((S,X_S),(T,X_T)).\]
Consider two functions $\varphi$ and $\psi$ such that
\[E[\varphi(S,X_S)]=E[\psi(T,X_T)]=0, \Var(\varphi(S,X_S))<\infty\text{ and }\Var(\varphi(T,X_{T}))<\infty.\]
For $s\subset[n]$, define a function $\varphi_{s}$ by
\[\varphi_{s}(x_{s})=\varphi(s,x_{s}).\]
Note that $\varphi_{s}$ depends solely on $(x_j)_{j\in s}$. Similarly, we define $\psi_{t}$ for $t\subset[n]$. Then by the independence between $(S,T)$ and $X$, we have that
\begin{align}\label{eq: Var varphi S X_S}
\Var(\varphi(S,X_S))&=E[\varphi^2(S,X_S)]\notag\\
&=\sum_{s\subset[n]}P(S=s)E[\varphi_s^2(X_s)]\notag\\
&=\sum_{u\subset[n]}\sum_{s:s\supset u}P(S=s)E[\varphi_{s,u}^2(X_u)],
\end{align}
where
\[\varphi_{s,u}:=\overline{E}_{u}\varphi_{s}.\]
The last equality of \eqref{eq: Var varphi S X_S} comes from the orthogonality in the ANOVA decomposition
\[\varphi_s=\sum_{u:u\subset s}\overline{E}_{u}\varphi_{s}.\]
Similarly, we have that
\begin{equation}\label{eq: Var psi T X_T}
\Var(\psi(T,X_T))=\sum_{u\subset[n]}\sum_{t:t\supset u}P(T=t)E[\psi_{t,u}^2(X_u)],
\end{equation}
where
\[\psi_{t,u}:=\overline{E}_{u}\psi_{t}.\]
Moreover, we have that
\begin{align}\label{eq: Cov varphi S X_S psi T X_T 1}
\Cov(\varphi(S,X_S),\psi(T,X_T))=&E[\varphi(S,X_S)\psi(T,X_T)]\notag\\
=&\sum_{u\subset[n]}\sum_{(s,t):s\cap t\supset u}P(S=s,T=t)E[\varphi_{s,u}(X_u)\psi_{t,u}(X_u)]\notag\\
\leq &\sum_{u\subset[n]}\sum_{(s,t):s\cap t\supset u}P(S=s,T=t)\sqrt{E[\varphi_{s,u}^2(X_u)]}\sqrt{E[\psi_{t,u}^2(X_u)]}.
\end{align}

For $u\neq\emptyset$, we have that
\begin{multline}\label{eq: bound via r_u for non-empty u}
\sum_{(s,t):s\cap t\supset u}P(S=s,T=t)\sqrt{E[\varphi_{s,u}^2(X_u)]}\sqrt{E[\psi_{t,u}^2(X_u)]}\\
\leq r_u\sqrt{\sum_{s:s\supset u}P(S=s)E[\varphi_{s,u}^2(X_u)]}\sqrt{\sum_{t:t\supset u}P(T=t)E[\psi_{t,u}^2(X_u)]},
\end{multline}
where $r_u$ is the best constant $r$ for the inequality
\begin{equation}\label{eq: defn r_u}
\sum_{s,t}P(S=s,T=t)\alpha_s\beta_t\leq r\sqrt{\sum_{s}P(S=s)\alpha_s^2}\sqrt{\sum_{t}P(T=t)\beta_t^2}
\end{equation}
for all real $\alpha_s$ and $\beta_t$ such that $\alpha_{s}=0$ if $s$ does not contain $u$ and $\beta_{t}=0$ if $t$ does not contain $u$.

For $u=\emptyset$, we have that $\varphi_{s,u}=E[\varphi_s(X_s)]$ and $\psi_{t,u}=E[\psi_t(X_t)]$. Since $E[\varphi(S,X_S)]=0$ and $E[\psi(T,X_T)]=0$, we have that
\begin{equation*}
\sum_{s\subset[n]}P(S=s)E[\varphi_{s}(X_s)]=\sum_{t\subset[n]}P(T=t)E[\varphi_{t}(X_t)]=0.
\end{equation*}
Define $f(s)=E[\varphi_{s}(X_s)]$ and $g(t)=E[\varphi_{t}(X_t)]$. Then we have that
\[E[f(S)]=E[g(T)]=0, \Var(f(S))<\infty, \Var(g(T))<\infty.\]
By the definition of maximal correlation coefficients, we have that
\[\Cov(f(S),g(T))\leq R(S,T)\sqrt{\Var(f(S))}\sqrt{\Var(g(T))}.\]
For $u=\emptyset$, we have that $\varphi_{s,u}(X_u)=f(s)$, $\psi_{t,u}(X_u)=g(t)$ and
\begin{multline}\label{eq: bound via R S T for empty u}
\sum_{(s,t):s\cap t\supset u}P(S=s,T=t)\sqrt{E[\varphi_{s,u}^2(X_u)]}\sqrt{E[\psi_{t,u}^2(X_u)]}\\
\leq R(S,T)\sqrt{\sum_{s:s\supset u}P(S=s)E[\varphi_{s,u}^2(X_u)]}\sqrt{\sum_{t:t\supset u}P(T=t)E[\psi_{t,u}^2(X_u)]}.
\end{multline}

By \eqref{eq: Var varphi S X_S}, \eqref{eq: Var psi T X_T}, \eqref{eq: Cov varphi S X_S psi T X_T 1}, \eqref{eq: bound via r_u for non-empty u}, \eqref{eq: bound via R S T for empty u} and the Cauchy-Schwarz inequality, we have that
\begin{align}
\eqref{eq: Cov varphi S X_S psi T X_T 1}&\leq\max(R(S,T),\max(r_u:u\neq\emptyset))\notag\\
&\quad\times\sum_{u\subset[n]}\sqrt{\sum_{s:s\supset u}P(S=s)E[\varphi_{s,u}^2(X_u)]}\sqrt{\sum_{t:t\supset u}P(T=t)E[\psi_{t,u}^2(X_u)]}\notag\\
&\leq\max(R(S,T),\max(r_u:u\neq\emptyset))\notag\\
&\quad\times\sqrt{\sum_{u\subset[n]}\sum_{s:s\supset u}P(S=s)E[\varphi_{s,u}^2(X_u)]}\sqrt{\sum_{u\subset[n]}\sum_{t:t\supset u}P(T=t)E[\psi_{t,u}^2(X_u)]}\notag\\
&=\max(R(S,T),\max(r_u:u\neq\emptyset))\sqrt{\Var(\varphi(S,X_S))}\sqrt{\Var(\psi(T,X_T))}.
\end{align}
As $\varphi$ and $\psi$ are arbitrary, we conclude that
\[R((S,X_S),(T,X_T))\leq \max(R(S,T),\max(r_u:u\neq\emptyset)).\]
By the definition of $r_u$, for $u\subset\tilde{u}$, we have that $r_u\geq r_{\tilde{u}}$. Hence, we have that
\[\max(r_u:u\neq\emptyset)=\max(r_{\{j\}}:j\in[n]).\]

It remains to prove the reverse inequality \[R((S,X_S),(T,X_T))\geq\max(R(S,T),\max(r_{\{j\}}:j\in[n])).\]
For $j\in[n]$, we take
\[\varphi(S,X_S)=\alpha_{S}h(X_{j})1_{j\in S},\quad\psi(T,X_T)=\beta_{T}h(X_{j})1_{j\in T}\]
for some measurable function $h$ such that $E[h(X_{j})]=0$, $\Var(h(X_{j}))<\infty$ and for $\alpha_s$ and $\beta_t$ achieving the best constant in the definition of $r_{\{j\}}$ (see \eqref{eq: defn r_u}). It is straightforward to check that for this particular choice of functions, we have that
\[\Cov(\varphi(S,X_S),\psi(T,X_T))=r_{\{j\}}\sqrt{\Var(\varphi(S,X_S))\Var(\psi(T,X_T))}.\]
Hence, $R((S,X_S),(T,X_T))\geq r_{\{j\}}$ for each $j\in[n]$. On the other hand, it is clear that
\[R((S,X_S),(T,X_T))\geq R(S,T).\]
Hence, the reverse inequality is proved.

When $S$ and $T$ are independent, we have $R(S,T)=0$. By writing $\widetilde{\alpha}_s=\sqrt{P(S=s)}\alpha_s$ and $\widetilde{\beta}_t=\sqrt{P(T=t)}\beta_t$, we see that $r_{\{j\}}$ is precisely the spectral norm of the rank-one matrix $(M_{st})_{s,t:j\in s\cap t}$, where $M_{st}=\sqrt{P(S=s)P(T=t)}$. Hence, we have that
\[r_{\{j\}}=\sqrt{P(j\in S)P(j\in T)}.\]

\subsection{Proof of Theorem~\ref{thm: R S T}}
In this subsection, we prove Theorem~\ref{thm: R S T}. We will use the alternative definition \eqref{eq: alt defn R} of the maximal correlation coefficient, the mathematical induction and the following lemma.
\begin{lemma}\label{lem: R S T n-1}
Let $T$ be a uniform subset of $[n]=\{1,2,\ldots,n\}$ of size $n-1$. Given $T$, the random set $S$ is a uniform subset of $T$ of size $k$. Then we have that
\[R(S,T)=\sqrt{\frac{k}{(n-1)(n-k)}}.\]
\end{lemma}
\begin{proof}
Let $\varphi$ be a measurable function such that $E[\varphi(T)]=0$ and $\Var(\varphi(T))<\infty$. Then we have that
\[E[\varphi(T)|S=s]=\sum_{i=1}^{n}\varphi(\{i\}^{c})P(T=\{i\}^c|S=s)\]
and that
\begin{align*}
E[(E[\varphi(T)|S])^2]&=\sum_{s\subset[n]}P(S=s)\left(\sum_{i=1}^{n}\varphi(\{i\}^c)P(T=\{i\}^c|S=s)\right)^2\\
&=\sum_{s\subset[n]}\sum_{i=1}^{n}\sum_{j=1}^{n}\varphi(\{i\}^c)\varphi(\{j\}^c)P(S=s)\\
&\quad\quad\quad\quad\quad\quad\times P(T=\{i\}^c|S=s)P(T=\{j\}^c|S=s)\\
&=\sum_{i=1}^{n}\sum_{j=1}^{n}\varphi(\{i\}^c)\varphi(\{j\}^c)A_{ij},
\end{align*}
where
\begin{align*}
A_{ij}&=\sum_{s\subset[n]}P(S=s)P(T=\{i\}^{c}|S=s)P(T=\{j\}^{c}|S=s)\\
&=\frac{1}{(n-k)^2}P(i\notin S,j\notin S)\\
&=\left\{\begin{array}{ll}
\frac{1}{(n-k)n}, & i=j,\\
\frac{(n-k-1)}{n(n-1)(n-k)}, & i\neq j.
\end{array}\right.
\end{align*}
The only element outside of $T$ is uniformly distributed in $[n]$. Hence, the condition $E[\varphi(T)]=0$ and $\Var(\varphi(T))<\infty$ is equivalent to
\[\sum_{i=1}^{n}\varphi(\{i\}^c)=0\]
and $\varphi(\{i\}^c)\in\mathbb{R}$ for each $i\in[n]$. Hence, by \eqref{eq: alt defn R}, we have that \[R^2(S,T)=\sup\left\{\sum_{i,j=1}^{n}A_{ij}\varphi(\{i\}^c)\varphi(\{j\}^c):\sum_{i=1}^{n}\varphi(\{i\}^c)=0,\sum_{i=1}^{n}(\varphi(\{i\}^c))^2=n\right\}.\]
Note that $A=(A_{ij})_{i,j\in[n]}$ is a real symmetric matrix with positive elements. Hence, the eigenvalues of $A$ are real. By the Perron-Frobenius theorem, $A$ has a simple eigenvalue $1/n$, and the moduli of other eigenvalues are strictly less than $1/n$. The eigenvector associated with $1/n$ is the column vector $\bm{1}=(1,1,\ldots,1)^{T}$. And the eigenvectors associated with other eigenvalues are perpendicular to $\bm{1}$. Hence, we get that
\[R^{2}(S,T)=n\lambda_2(A),\]
where $\lambda_2(A)$ is the second largest eigenvalue of $A$. Note that \[A=\frac{k}{n(n-1)(n-k)}I+\frac{n-k-1}{n(n-1)(n-k)}\bm{1}\bm{1}^{T},\]
where $I$ is the identity matrix. Hence, using a bit of linear algebra, we find that $\lambda_2(A)=\frac{k}{n(n-1)(n-k)}$. Hence, $R(S,T)=\sqrt{n\lambda_2(A)}=\sqrt{\frac{k}{(n-1)(n-k)}}$.
\end{proof}

Next, we give the proof of Theorem~\ref{thm: R S T} by induction on $n$. For $n=1$, $S$ and $T$ are independent. Hence, $R(S,T)=0$ and Theorem~\ref{thm: R S T} holds. Suppose that Theorem~\ref{thm: R S T} holds for $n\leq N$, and consider $n=N+1$. If $m=n$ or $m=0$, then $T$ is independent of $S$. Hence, $R(S,T)=0$ and Theorem~\ref{thm: R S T} holds. Therefore, we may assume that $0<m<n$ in the following. Let $U$ be a random subset of $[n]$ with cardinality $n-1$ that contains $T$. Moreover, assume that $U$ is uniform given $T$. In this way, we have $S\subset T\subset U$, and $(S,T,U)$ is uniformly distributed. Let $\varphi$ be a measurable function such that $E[\varphi(S)]=0$ and $E[(\varphi(S))^2]=1$. Since $(S,T,U)$ is Markov, we have that
\[E[\varphi(S)|T]=E[\varphi(S)|T,U].\]
Write $\varphi(S)=g(U)+f(S,U)$, where
\[g(U)=E[\varphi(S)|U].\]
Then we have that $E[f(S,U)|U]=0$ and $E[\varphi(S)|T,U]=g(U)+E[f(S,U)|T,U]$. Note that
\begin{align*}
(E[\varphi(S)|T])^2&=(E[\varphi(S)|T,U])^2\\
&=(g(U))^2+(E[f(S,U)|T,U])^2+2g(U)E[f(S,U)|T,U].
\end{align*}
Since $E[E[f(S,U)|T,U]|U]=E[f(S,U)|U]=0$, we have that
\begin{equation}\label{eq: EEvarphiST2U}
E[(E[\varphi(S)|T])^2|U]=(g(U))^2+E[(E[f(S,U)|T,U])^2|U].
\end{equation}
Since $E[f(S,U)|U]=0$ and $\Var(f(S,U)|U)<\infty$, by \eqref{eq: alt defn R}, we have that
\begin{equation}\label{eq: EEfSUTU2U}
E[(E[f(S,U)|T,U])^2|U]\leq (R(S,T|U))^2E[(f(S,U))^2|U],
\end{equation}
where $R(S,T|U)$ is the maximal correlation coefficient of the conditional distribution of $(S,T)$ given $U$. By the induction hypothesis, we have
\begin{equation}\label{eq: RSTU2}
(R(S,T|U))^2=\frac{k(n-1-m)}{m(n-1-k)}.
\end{equation}
Note that
\begin{equation}\label{eq: EfSU2U}
E[(f(S,U))^2|U]=E[(\varphi(S)-g(U))^2|U]=E[(\varphi(S))^2|U]-(g(U))^2.
\end{equation}
By \eqref{eq: EEvarphiST2U}, \eqref{eq: EEfSUTU2U}, \eqref{eq: RSTU2} and \eqref{eq: EfSU2U}, we obtain that
\begin{equation*}
E[(E[\varphi(S)|T])^2|U]\leq \frac{k(n-1-m)}{m(n-1-k)}E[(\varphi(S))^2|U]+\left(1-\frac{k(n-1-m)}{m(n-1-k)}\right)(g(U))^2.
\end{equation*}
By taking the expectation on both sides, we get that
\begin{equation}\label{eq: EEvarphiST2}
E[(E[\varphi(S)|T])^2]\leq \frac{k(n-1-m)}{m(n-1-k)}E[(\varphi(S))^2]+\left(1-\frac{k(n-1-m)}{m(n-1-k)}\right)E[(g(U))^2].
\end{equation}
Note that $g(U)=E[\varphi(S)|U]$. Hence, by \eqref{eq: alt defn R} and Lemma~\ref{lem: R S T n-1}, we get that
\begin{equation}\label{eq: EgU2}
E[(g(U))^2]\leq (R(S,U))^2E[(\varphi(S))^2]=\frac{k}{(n-1)(n-k)}E[(\varphi(S))^2].
\end{equation}
Finally, combining \eqref{eq: EEvarphiST2} and \eqref{eq: EgU2}, using the identity \[\frac{k(n-1-m)}{m(n-1-k)}+\left(1-\frac{k(n-1-m)}{m(n-1-k)}\right)\frac{k}{(n-1)(n-k)}=\frac{k(n-m)}{m(n-k)},\]
we get that
\[E[(E[\varphi(S)|T])^2]\leq \frac{k(n-m)}{m(n-k)}E[(\varphi(S))^2].\]
Since this holds for all measurable $\varphi$ with $E[\varphi(S)]=0$ and $\Var(\varphi(S))=1$, we conclude that
\[(R(S,T))^2\leq \frac{k(n-m)}{m(n-k)}.\]
To prove the reverse inequality $R(S,T)\geq \sqrt{\frac{k(n-m)}{m(n-k)}}$, we take $\varphi(S)=1_{1\in S}$ and $\psi(T)=1_{1\in T}$. Then we have that
\[R(S,T)\geq\rho(\varphi(S),\psi(T))=\sqrt{\frac{k(n-m)}{m(n-k)}}.\]
%\begin{remark}
%For $1\leq\ell\leq n$, let $\xi=|[\ell]\cap S|$ and $\eta=|[\ell]\cap T^{c}|$. Then $(\xi,\eta)$ follows the multivariate hypergeometric distribution. More precisely, we have that
%\[P(\xi=i,\eta=j)=\frac{\binom{k}{i}\binom{n-m}{j}\binom{m-k}{\ell-i-j}}{\binom{n}{\ell}},\quad i\geq 0,j\geq 0,i+j\leq \ell.\]
%Hence, we have that
%\[R(\xi,\eta)\leq R(S,T)=\sqrt{\frac{k(n-m)}{m(n-k)}}.\]
%The above upper bound agrees with the maximal correlation coefficient for multivariate hypergeometric distribution, see \cite[Eq. (6.12)]{CsakiFischerMR0166833}.
%\end{remark}

\subsection{Proof of Corollary~\ref{cor: generalization of Yu}}
In this subsection, we prove Corollary~\ref{cor: generalization of Yu}. In the proof, we use Theorem~\ref{thm: R (S,X_S) (T,X_T)}, Theorem~\ref{thm: R S T} and Lemma~\ref{lem: submultiplicative} (submultiplicative property).

Let $(S,T)$ be uniformly distributed with the constraints that
\[S\subset [n], T\subset [n],|S|=m, |T|=n-\ell\text{ and }|S\cap T|=m-\ell.\]
Suppose that $(S,T)$ is independent of $X_1,X_2,\ldots,X_n$. Let $U=S\cap T$. Since $X_1,X_2,\ldots,X_n$ are i.i.d., $(\sum_{i\in S}\delta_{X_i},\sum_{i\in T}\delta_{X_i})$ has the same distribution as $(\sum_{i=1}^{m}\delta_{X_i},\sum_{i=\ell+1}^{n}\delta_{X_i})$. Moreover, $\sum_{i\in S}\delta_{X_i}$ and $\sum_{i\in T}\delta_{X_i}$ are conditionally independent given $\sum_{i\in U}\delta_{X_i}$. Hence, by Lemma~\ref{lem: submultiplicative}, we have that
\begin{align}\label{eq: RSXSTXTRUXUSXSUXUTXT}
R\left(\sum_{i=1}^{m}\delta_{X_i},\sum_{i=\ell+1}^{n}\delta_{X_i}\right)&=R\left(\sum_{i\in S}\delta_{X_i},\sum_{i\in T}\delta_{X_i}\right)\notag\\
&\leq R\left(\sum_{i\in S}\delta_{X_i},\sum_{i\in U}\delta_{X_i}\right)R\left(\sum_{i\in U}\delta_{X_i},\sum_{i\in T}\delta_{X_i}\right)\notag\\
&\leq R((U,X_{U}),(S,X_{S}))R((U,X_U),(T,X_{T})).
\end{align}
We explicitly calculate $R((U,X_{U}),(S,X_{S}))$ in the following lemma.
\begin{lemma}\label{lem: calculation of rj special}
Consider independent non-degenerate random variables $X_1,X_2,\ldots,X_n$ taking values in a general measurable space. Let $(U,S)$ be uniformly distributed with the constraints that $|U|=a$, $|S|=b$ and $U\subset S\subset [n]$. Suppose that $(U,S)$ is independent of $X_1,X_2,\ldots,X_n$. Then we have that
\begin{equation}\label{eq: R U X_U S X_S}
R((U,X_U),(S,X_S))=\sqrt{a/b}.
\end{equation}
\end{lemma}
\begin{remark}
Lemma~\ref{lem: calculation of rj special} is an extension of the Dembo-Kagan-Shepp inequality for independent random variables $X_1,X_2,\ldots,X_n$ with possibly different distributions.
\end{remark}
\begin{proof}
By Theorem~\ref{thm: R (S,X_S) (T,X_T)}, we have that
\[R((U,X_U),(S,X_S))=\max(R(U,S),\max(r_j:j\in[n])),\]
where $r_j$ is the best constant $r$ in the following inequality
\begin{equation}
\sum_{u,s:j\in u\cap s}P(U=u,S=s)\alpha_u\beta_s\leq r\sqrt{\sum_{u:j\in u}P(U=u)\alpha_u^2}\sqrt{\sum_{s:j\in s}P(S=s)\beta_s^2}
\end{equation}
for arbitrary real constants $\alpha_u$ and $\beta_s$. By Theorem~\ref{thm: R S T}, we see that
\[R(U,S)=\sqrt{\frac{a(n-b)}{b(n-a)}}\leq\sqrt{\frac{a}{b}}.\]
Hence, it suffices to prove that
\begin{equation}
r_{j}=\sqrt{\frac{a}{b}}.
\end{equation}
By changing the variables $\widetilde{\alpha}_u=\sqrt{P(U=u)}\alpha_u$ and $\widetilde{\beta}_{s}=\sqrt{P(S=s)}\beta_s$, we find that $r_{j}$ is precisely the spectral norm of the matrix $A=(A_{us})_{u,s:j\in u\cap s}$, where
\[A_{us}=\frac{P(U=u,S=s)}{\sqrt{P(U=u)P(S=s)}}\]
with the convention that $0/0=0$. Let $B=AA^{*}$, where $A^{*}$ denotes the transpose of $A$. Then $B$ is real, symmetric and positively definite. Moreover, $r_j^2$ is exactly the maximal eigenvalue of $B$. Note that
\begin{align*}
B_{uv}&=\sum_{s:j\in s}A_{us}A_{vs}\\
&=\sum_{s:j\in s}\frac{P(U=u,S=s)P(U=v,S=s)}{\sqrt{P(U=u)P(U=v)}P(S=s)}\\
&=\frac{\sqrt{P(U=u)}}{\sqrt{P(U=v)}}\sum_{s:j\in s}P(S=s|U=u)P(U=v|S=s).
\end{align*}
Then we see that $B$ is similar to $C$, where
\[C_{uv}=\sum_{s:j\in s}P(S=s|U=u)P(U=v|S=s)\]
with the convention that the conditional probability is zero if it is not well-defined. Hence, $r_j^2$ is equal to the maximal eigenvalue of $C$. Note that $P(j\in U|S=s)=a/b$ for $j\in s$ and $P(j\in S|U=u)=1$ for $j\in u$. Hence, we have that
\begin{align*}
\sum_{v:j\in v}C_{uv}&=\sum_{s:j\in s}\sum_{v:j\in v}P(S=s|U=u)P(U=v|S=s)\\
&=\sum_{s:j\in s}P(S=s|U=u)P(j\in U|S=s)\\
&=\frac{a}{b}\sum_{s:j\in s}P(S=s|U=u)\\
&=\frac{a}{b}P(j\in S|U=u)\\
&=\frac{a}{b}.
\end{align*}
Note that $C$ is a matrix with non-negative elements such that the sum of each row is the constant $a/b$. By the Perron-Frobenius theorem, the maximal eigenvalue of $C$ is precisely $a/b$. Thus, $r_j=\sqrt{a/b}$, and the proof is complete.
\end{proof}
By \eqref{eq: RSXSTXTRUXUSXSUXUTXT} and \eqref{eq: R U X_U S X_S}, we obtain the upper bound
\[R\left(\sum_{i=1}^{m}\delta_{X_i},\sum_{i=\ell+1}^{n}\delta_{X_i}\right)\leq \sqrt{\frac{m-\ell}{m}}\sqrt{\frac{m-\ell}{n-\ell}}=\frac{m-\ell}{\sqrt{m(n-\ell)}}.\]
For the lower bound, by the Dembo-Kagan-Shepp-Yu equality (\cite[Theorem~4.1]{YuMR2422962}), we have that
\[R\left(\sum_{i=1}^{m}\delta_{X_i},\sum_{i=\ell+1}^{n}\delta_{X_i}\right)\geq R\left(\sum_{i=1}^{m}\varphi(X_i),\sum_{i=\ell+1}^{n}\varphi(X_i)\right)=\frac{m-\ell}{\sqrt{m(n-\ell)}}\]
for some non-degenerate real-valued measurable function $\varphi$.
\begin{remark}
Under the assumption of Theorem~\ref{thm: R (S,X_S) (T,X_T)}, using similar arguments as in Lemma~\ref{lem: calculation of rj special}, we have the following observation about $r_j$: If $P(j\in S|U=u)$ does not depend on $u$ when $j\in u$ and $P(j\in U|S=s)$ does not depend on $s$ when $j\in s$, then we have that
\begin{equation*}
r_j^2=P(j\in S|U=u)P(j\in U|S=s).
\end{equation*}
\end{remark}

\subsection{Proof of Theorem~\ref{thm: applications in information theory}}
In this subsection, we prove Theorem~\ref{thm: applications in information theory} .

For a non-empty subset $s\subset[n]$, we define $U_{s}=\sum_{i\in s}X_i$ and denote by $f_{s}(u)$ the density of $U_{s}$. Let $\rho_{s}(u)=f'_{s}(u)/f_s(u)$ be the score function of $U_s$. Then the Fisher information
\[I(U_s)=E[\rho^2_{s}(U_s)].\]
For two nested subsets $s\subset t$, we have that
\[E[\rho_{s}(U_s)|U_{t}]=\rho_{t}(U_t)\]
by \cite[Lemma~1]{MadimanBarronMR2319376} (i.e. the convolution identity for scores). Define
\[\varphi(s,x_s)=\lambda_s\rho_s(u_s),\]
where $u_s=\sum_{i\in s}x_i$. Then we have that
\[E[\varphi(S,X_S)|T=t,U_{t}]=\left(\sum_{s\subset[n]}P(S=s|T=t)\lambda_s\right)\rho_{t}(U_t)=\mu_t\rho_{t}(U_t).\]
Hence, we get that
\[E[(E[\varphi(S,X_S)|T=t,U_{t}])^2|T=t]=\mu_t^2I(U_t).\]
Consequently, we have that
\[E[(E[\varphi(S,X_S)|T,U_{T}])^2]=E\left[E\left[(E[\varphi(S,X_S)|T,U_{T}])^2|T\right]\right]=\sum_{t\subset[n]}P(T=t)\mu_t^2I(U_t).\]
Note that
\[E[(\varphi(S,X_S))^2]=\sum_{s\subset[n]}P(S=s)\lambda_s^2I(U_s).\]
By \eqref{eq: alt defn R} and the fact that $(T,U_{T})$ is a measurable function of $(T,X_{T})$, we have that
\[E[(E[\varphi(S,X_S)|T,U_{T}])^2]\leq R^2((T,U_{T}),(S,X_{S}))E[(\varphi(S,X_S))^2]\leq R^2E[(\varphi(S,X_S))^2].\]
Therefore, we obtain that
\[\sum_{t\subset[n]}P(T=t)\mu_t^2I(U_t)\leq R^2\sum_{s\subset[n]}P(S=s)\lambda_s^2I(U_s).\]

\subsection{Examples}
Firstly, we consider an example studied in \cite{BucherStaudMR4835998}. They calculated the maximal correlation for the bivariate Marshall-Olkin exponential distribution. We state their result in the following theorem:
\begin{theorem}[B\"{u}cher-Staud]
Let $W_1,W_2$ and $W_3$ be independent exponential random variables. The parameter of $W_i$ is $\lambda_i>0$ for $i=1,2,3$. Let $V_1=\min(W_1,W_3)$ and $V_2=\min(W_2,W_3)$. Then we have that
\[R(V_1,V_2)=\frac{\lambda_3}{\sqrt{(\lambda_1+\lambda_3)(\lambda_2+\lambda_3)}}.\]
\end{theorem}
We will give an alternative proof of the upper bound
\begin{equation}\label{eq: BS upper bound}
R(V_1,V_2)\leq\frac{\lambda_3}{\sqrt{(\lambda_1+\lambda_3)(\lambda_2+\lambda_3)}}
\end{equation}
using Corollary~\ref{cor: generalization of Yu} and Lemma~\ref{lem: lower semi-continuity}.

Let $1\leq\ell+1\leq m\leq n$. Let us take independent random variables $X_1,X_2,\ldots,X_n$. Each $X_i$ is an exponential random variable with parameter $\lambda$. Let
\[W_1=\min_{i=1,2,\ldots,\ell}X_i, W_2=\min_{i=m+1,m+2,\ldots,n}X_i\text{ and }W_3=\min_{i=\ell+1,\ell+2,\ldots,m}X_i.\]
Then $W_1$, $W_2$ and $W_3$ are independent. Moreover, $W_1$ is an exponential random variable with parameter $\ell\lambda$, $W_2$ is an exponential random variable with parameter $(n-m)\lambda$, and $W_3$ is an exponential random variable with parameter $(m-\ell)\lambda$. Let $V_1=\min(W_1,W_3)$ and $V_2=\min(W_2,W_3)$. By Corollary~\ref{cor: generalization of Yu}, we have that
\[R(V_1,V_2)\leq R\left(\sum_{i=1}^{m}\delta_{X_i},\sum_{j=\ell+1}^{n}\delta_{X_j}\right)=\frac{m-\ell}{\sqrt{m(n-\ell)}}.\]
Hence, \eqref{eq: BS upper bound} holds with $\lambda_1=\ell\lambda$, $\lambda_2=(n-m)\lambda$ and $\lambda_3=(m-\ell)\lambda$. Therefore, \eqref{eq: BS upper bound} holds for rational $\lambda_1,\lambda_2$ and $\lambda_3$. For general $\lambda_1,\lambda_2$ and $\lambda_3$, take $\lambda^{(N)}_{i}=[N\lambda_i]/N$ for $i=1,2,3$ and $N\geq 1$. For each $N\geq 1$, the corresponding random variables are $W^{(N)}_{1},W^{(N)}_{2},W^{(N)}_{3},V^{(N)}_{1}$ and $V^{(N)}_{2}$. As $N\to\infty$, $(V^{(N)}_{1},V^{(N)}_{2})$ converges in distribution to $(V_1,V_2)$. By Lemma~\ref{lem: lower semi-continuity}, we have that
\begin{align*}
R(V_1,V_2)&\leq\liminf_{N\to\infty}R(V^{(N)}_{1},V^{(N)}_{2})\\ &\leq\liminf_{N\to\infty}\frac{\lambda_3^{(N)}}{\sqrt{(\lambda_1^{(N)}+\lambda_3^{(N)})(\lambda_2^{(N)}+\lambda_3^{(N)})}}\\ &=\frac{\lambda_3}{\sqrt{(\lambda_1+\lambda_3)(\lambda_2+\lambda_3)}}.
\end{align*}
Using similar arguments, we obtain the following upper bound:
\begin{proposition}\label{prop: Yu replace sum by min}
Let $X_1,X_2,\ldots,X_n$ be i.i.d. real-valued random variables. Let $1\leq \ell+1\leq m\leq n$. Then we have that
\begin{equation}\label{eq: Yu replace sum by min}
R\left(\min_{i:1\leq i\leq m}X_i,\min_{j:\ell+1\leq j\leq n}X_j\right)\leq\frac{m-\ell}{\sqrt{m(n-\ell)}}.
\end{equation}
\end{proposition}

\section{Open problems}\label{sect: open problems}

\begin{enumerate}
  \item Let $(X_t,Y_t)_{t\in[0,1]}$ be a two-dimensional L\'{e}vy bridge. Is there an expression for the maximal correlation coefficient $R((X_t)_{t\in [0,1]},(Y_t)_{t\in[0,1]})$?
  \item By \cite[Corollary~2.2]{BucherStaudMR4835998}, the upper bound in \eqref{eq: Yu replace sum by min} is sharp if $X_i$ follows an exponential distribution. Is it also sharp for other distributions?
\end{enumerate}

\appendix

\section{Proof of Lemma~\ref{lem: submultiplicative}}\label{appendix: A}
Lemma~\ref{lem: submultiplicative} is derived from the fact that the operator norm of the composition of two operators is not greater than the product of the operator norms of these two operators. In the proof, the Markov property of $(X,Y,Z)$ is crucially used.
\begin{proof}[Proof of Lemma~\ref{lem: submultiplicative}]
By the conditional independence of $X$ and $Z$ given $Y$, the following diagram commutes:
\begin{equation}
\xymatrix{
L^2_0(X) \ar[dd]^{\pi_1} \ar[rd]^{\pi_2} \ar@/^1pc/[rrd]^{\pi_3} & & \\
& L^2_0(Y) \ar[ld]^{\pi_4} \ar@{^(->}[r]^{\iota} & L^2_0(Y,Z) \ar@/^1pc/[lld]^{\pi_5}\\
L^2_0(Z)  & &
}
\end{equation}
where $\pi_1,\pi_2, \pi_3, \pi_4$ and $\pi_5$ are orthogonal projections and $\iota$ is the injection. Specifically, $\pi_1:W\mapsto E(W|Z)$, $\pi_2:W\mapsto E(W|Y)$, $\pi_3:W\mapsto E(W|Y,Z)$, $\pi_4:W\mapsto E(W|Z)$ and $\pi_5:W\mapsto E(W|Z)$ are conditional expectation operators. Hence, we have that
\[R(X,Z)=\|\pi_1\|=\|\pi_4\circ\pi_2\|\leq \|\pi_4\|\|\pi_2\|=R(Y,Z)R(X,Y),\]
where $\|\cdot\|$ denotes the operator norm.
\end{proof}

\section{Proof of Lemma~\ref{lem: lower semi-continuity}}\label{appendix: B}
To prove Lemma~\ref{lem: lower semi-continuity}, we need to use the classical result that $L^p$ functions can be approximated by bounded continuous functions. We give a precise statement as follows:
\begin{lemma}\label{lem: approximation of Lp functions by bounded continuous functions}
Consider a metric space S with the Borel $\sigma$-field $\mathcal{S}$, a bounded measure $\mu$ on $(S,\mathcal{S})$ and a constant $p > 0$. Then the bounded continuous functions on S are dense in $L^p(S,\mathcal{S},\mu)$. Thus, for any $f\in L^{p},$ there exist bounded continuous functions $f_{1},f_{2},\ldots:S\to\mathbb{R}$ with $\|f_n-f\|_p\to 0$.
\end{lemma}
The above lemma is precisely Lemma~1.37 in \cite{KallenbergMR4226142}.
\begin{proof}[Proof of Lemma~\ref{lem: lower semi-continuity}]
Let $(X_n,Y_n)$ be a sequence of random variables taking values in the measurable space $(S,\mathcal{S})$. Suppose that $(X_n,Y_n)$ converges weakly to $(X,Y)$ as $n\to\infty$. It suffices to show that
\begin{equation*}
R(X,Y)\leq\liminf_{n\to\infty}R(X_n,Y_n).
\end{equation*}
For any $\varepsilon>0$, by the definition of the maximal correlation coefficient, there exist $\varphi$ and $\psi$ such that
\[R(X,Y)\leq\rho(\varphi(X),\psi(Y))+\varepsilon\]
and that $\Var(\varphi(X))>0$ and $\Var(\psi(Y))>0$. By Lemma~\ref{lem: approximation of Lp functions by bounded continuous functions} with $p=2$, for any $\delta>0$, there exist bounded continuous functions $f$ and $g$ such that
\begin{equation*}
E[(f(X)-\varphi(X))^2]<\delta, E[(g(Y)-\psi(Y))^2]<\delta.
\end{equation*}
Since $\Var(\varphi(X))>0$ and $\Var(\psi(Y))>0$, for sufficiently small $\delta>0$, we have that $\Var(f(X))>0$, $\Var(g(Y))>0$ and $\rho(\varphi(X),\psi(Y))\leq\rho(f(X),g(Y))+\varepsilon$. By weak convergence of $(X_n,Y_n)$ towards $(X,Y)$, we have that
\[\lim_{n\to\infty}\rho(f(X_n),g(Y_n))=\rho(f(X),g(Y)).\]
By the definition of maximal correlation coefficients,
\[\rho(f(X_n),g(Y_n))\leq R(X_n,Y_n).\]
Hence, for any $\varepsilon>0$, we have that
\[\liminf_{n\to\infty}R(X_n,Y_n)\geq\liminf_{n\to\infty}\rho(f(X_n),g(Y_n))=\rho(f(X),g(Y))\geq R(X,Y)-2\varepsilon.\]
By taking $\varepsilon\to 0$, we obtain the desired result.
\end{proof}

\section{Proof of Remark~\ref{rem: MB implies DKS}}\label{appendix: C}
In this part, we prove Remark~\ref{rem: MB implies DKS}.

Let $X_1,X_2,\ldots,X_n$ be i.i.d. real-valued random variables. For fixed $m\in[n]$, let $T$ be a uniform random subset of $\{1,2,\ldots,n\}$ with size $m$. For $i=1,2,\ldots,n$, define
\[Y_i=\left\{\begin{array}{ll}
X_i, & \text{if }i\in T,\\
\partial, & \text{otherwise},
\end{array}\right.\]
where $\partial$ is a special point outside of $\mathbb{R}$. By the Madiman-Barron inequality, for $n\geq 1$,
\[R((X_1,X_2,\ldots,X_n),(Y_1,Y_2,\ldots,Y_n))\leq\sqrt{m/n}.\]
Then $(\sum_{i=1}^{n}X_i,\sum_{i=1}^{n}Y_i1_{Y_i\neq\partial})$ has the same joint distribution as $(S_n,S_m)$ in \eqref{eq: BDKS identity}. Hence, we see that
\[R(S_n,S_m)\leq R((X_1,X_2,\ldots,X_n),(Y_1,Y_2,\ldots,Y_n))\leq\sqrt{m/n}.\]
Analogous results also hold for i.i.d. random vectors $X_1,X_2,\ldots,X_n$.

\section*{Acknowledgement}
The authors would like to thank the anonymous referees for their careful reading and valuable suggestions. We are grateful to Professor Renming Song for his careful reading and helpful suggestions to improve the English writing. Finally, we are grateful to Professor Elton P. Hsu who drew our attention to the problem of maximal correlations.

\bibliographystyle{alpha}
%\bibliography{max-corr}

\begin{thebibliography}{LBCnM06}

\bibitem[ABBN04]{ArtsteinBallBartheNaorMR2083473}
Shiri Artstein, Keith~M. Ball, Franck Barthe, and Assaf Naor.
\newblock Solution of {S}hannon's problem on the monotonicity of entropy.
\newblock {\em J. Amer. Math. Soc.}, 17(4):975--982, 2004.

\bibitem[AG76]{AhlswedeGacsMR424401}
Rudolf Ahlswede and P\'{e}ter G\'{a}cs.
\newblock Spreading of sets in product spaces and hypercontraction of the
  {M}arkov operator.
\newblock {\em Ann. Probability}, 4(6):925--939, 1976.

\bibitem[App09]{ApplebaumMR2512800}
David Applebaum.
\newblock {\em L\'evy processes and stochastic calculus}, volume 116 of {\em
  Cambridge Studies in Advanced Mathematics}.
\newblock Cambridge University Press, Cambridge, second edition, 2009.

\bibitem[BDK05]{BrycDemboKaganMR2141340}
W.~Bryc, A.~Dembo, and A.~Kagan.
\newblock On the maximum correlation coefficient.
\newblock {\em Theory of Probability \& Its Applications}, 49(1):132--138,
  2005.

\bibitem[BF85]{BreimanMR0803258}
Leo Breiman and Jerome~H. Friedman.
\newblock Estimating optimal transformations for multiple regression and
  correlation.
\newblock {\em J. Amer. Statist. Assoc.}, 80(391):580--619, 1985.
\newblock With discussion and with a reply by the authors.

\bibitem[Bil99]{BillingsleyMR1700749}
Patrick Billingsley.
\newblock {\em Convergence of probability measures}.
\newblock Wiley Series in Probability and Statistics: Probability and
  Statistics. John Wiley \& Sons, Inc., New York, second edition, 1999.
\newblock A Wiley-Interscience Publication.

\bibitem[BS25]{BucherStaudMR4835998}
Axel B\"ucher and Torben Staud.
\newblock On the maximal correlation coefficient for the bivariate {M}arshall
  {O}lkin distribution.
\newblock {\em Statist. Probab. Lett.}, 219:Paper No. 110323, 4, 2025.

\bibitem[CF60]{CsakiFischerMR0126952}
P\'{e}ter Cs\'{a}ki and J\'{a}nos Fischer.
\newblock Contributions to the problem of maximal correlation.
\newblock {\em Magyar Tud. Akad. Mat. Kutat\'{o} Int. K\"{o}zl.}, 5:325--337,
  1960.

\bibitem[CF63]{CsakiFischerMR0166833}
P\'{e}ter Cs\'{a}ki and J\'{a}nos Fischer.
\newblock On the general notion of maximal correlation.
\newblock {\em Magyar Tud. Akad. Mat. Kutat\'{o} Int. K\"{o}zl.}, 8:27--51,
  1963.

\bibitem[Cou16]{CourtadeMR3638565}
Thomas~A. Courtade.
\newblock Monotonicity of entropy and {F}isher information: a quick proof via
  maximal correlation.
\newblock {\em Commun. Inf. Syst.}, 16(2):111--115, 2016.

\bibitem[DKS01]{DemboKaganSheppMR1828509}
Amir Dembo, Abram Kagan, and Lawrence~A. Shepp.
\newblock Remarks on the maximum correlation coefficient.
\newblock {\em Bernoulli}, 7(2):343--350, 2001.

\bibitem[DY21]{DadounYoussefMR4255828}
Benjamin Dadoun and Pierre Youssef.
\newblock Maximal correlation and monotonicity of free entropy and of {S}tein
  discrepancy.
\newblock {\em Electron. Commun. Probab.}, 26:Paper No. 24, 10, 2021.

\bibitem[Geb41]{GebeleinMR7220}
Hans Gebelein.
\newblock Das statistische {P}roblem der {K}orrelation als {V}ariations und
  {E}igenwertproblem und sein {Z}usammenhang mit der {A}usgleichsrechnung.
\newblock {\em Z. Angew. Math. Mech.}, 21:364--379, 1941.

\bibitem[HLP88]{HardyLittlewoodPolyaMR944909}
G.~H. Hardy, J.~E. Littlewood, and G.~P\'olya.
\newblock {\em Inequalities}.
\newblock Cambridge Mathematical Library. Cambridge University Press,
  Cambridge, 1988.
\newblock Reprint of the 1952 edition.

\bibitem[It{\^o}56]{ItoMR0077017}
Kiyosi It{\^o}.
\newblock Spectral type of the shift transformation of differential processes
  with stationary increments.
\newblock {\em Trans. Amer. Math. Soc.}, 81:253--263, 1956.

\bibitem[JS03]{JacodShiryaevMR1943877}
Jean Jacod and Albert~N. Shiryaev.
\newblock {\em Limit theorems for stochastic processes}, volume 288 of {\em
  Grundlehren der mathematischen Wissenschaften [Fundamental Principles of
  Mathematical Sciences]}.
\newblock Springer-Verlag, Berlin, second edition, 2003.

\bibitem[KA12]{KamathAnantharam6483335}
Sudeep Kamath and Venkat Anantharam.
\newblock Non-interactive simulation of joint distributions: The
  {H}irschfeld-{G}ebelein-{R}\'{e}nyi maximal correlation and the
  hypercontractivity ribbon.
\newblock In {\em 2012 50th Annual Allerton Conference on Communication,
  Control, and Computing (Allerton)}, pages 1057--1064, 2012.

\bibitem[KA16]{KamathAnantharamMR3506743}
Sudeep Kamath and Venkat Anantharam.
\newblock On non-interactive simulation of joint distributions.
\newblock {\em IEEE Trans. Inform. Theory}, 62(6):3419--3435, 2016.

\bibitem[Kal21]{KallenbergMR4226142}
Olav Kallenberg.
\newblock {\em Foundations of modern probability}, volume~99 of {\em
  Probability Theory and Stochastic Modelling}.
\newblock Springer, Cham, third edition, [2021] \copyright2021.

\bibitem[Kun04]{KunitaMR2083711}
Hiroshi Kunita.
\newblock Representation of martingales with jumps and applications to
  mathematical finance.
\newblock In {\em Stochastic analysis and related topics in {K}yoto}, volume~41
  of {\em Adv. Stud. Pure Math.}, pages 209--232. Math. Soc. Japan, Tokyo,
  2004.

\bibitem[Lan57]{Lancaster}
H.~O. Lancaster.
\newblock {Some properties of the bivariate normal distribution considered in
  the form of a contingency table}.
\newblock {\em Biometrika}, 44(1-2):289--292, 06 1957.

\bibitem[LBCnM06]{LopezCastanoMR2207171}
Fernando L\'opez-Bl\'azquez and Antonia Casta\~no Mart\'inez.
\newblock Upper and lower bounds for the correlation ratio of order statistics
  from a sample without replacement.
\newblock {\em J. Statist. Plann. Inference}, 136(1):43--52, 2006.

\bibitem[LBSMn98]{LopezSalamancaMR1622148}
Fernando L\'opez-Bl\'azquez and Bego\~na Salamanca-Mi\~no.
\newblock An upper bound for the correlation ratio of records.
\newblock {\em Metrika}, 47(2):165--174, 1998.

\bibitem[LBSMn14]{LopezSalamancaMR3252649}
F.~L\'opez~Bl\'azquez and B.~Salamanca Mi\~no.
\newblock Maximal correlation in a non-diagonal case.
\newblock {\em J. Multivariate Anal.}, 131:265--278, 2014.

\bibitem[LWK94]{LiuMR1279653}
Jun~S. Liu, Wing~Hung Wong, and Augustine Kong.
\newblock Covariance structure of the {G}ibbs sampler with applications to the
  comparisons of estimators and augmentation schemes.
\newblock {\em Biometrika}, 81(1):27--40, 1994.

\bibitem[MB07]{MadimanBarronMR2319376}
Mokshay Madiman and Andrew Barron.
\newblock Generalized entropy power inequalities and monotonicity properties of
  information.
\newblock {\em IEEE Trans. Inform. Theory}, 53(7):2317--2329, 2007.

\bibitem[Nev92]{NevzorovMR1202799}
V.~B. Nevzorov.
\newblock A characterization of exponential distributions by correlations
  between records.
\newblock {\em Math. Methods Statist.}, 1:49--54, 1992.

\bibitem[Nov04]{NovakMR2089006}
S.~Y. Novak.
\newblock On {G}ebelein's correlation coefficient.
\newblock {\em Statist. Probab. Lett.}, 69(3):299--303, 2004.

\bibitem[PX13]{PapadatosXifaraMR3054093}
Nickos Papadatos and Tatiana Xifara.
\newblock A simple method for obtaining the maximal correlation coefficient and
  related characterizations.
\newblock {\em J. Multivariate Anal.}, 118:102--114, 2013.

\bibitem[R\'59]{RenyiMR0115203}
A.~R\'{e}nyi.
\newblock On measures of dependence.
\newblock {\em Acta Math. Acad. Sci. Hungar.}, 10:441--451 (unbound insert),
  1959.

\bibitem[Sar58]{SarmanovMR99095}
O.~V. Sarmanov.
\newblock Maximum correlation coefficient (non-symmetrical case).
\newblock {\em Dokl. Akad. Nauk SSSR}, 121:52--55, 1958.

\bibitem[SM85]{SzekelyMoriMR792800}
G.~J. Sz\'ekely and T.~F. M\'ori.
\newblock An extremal property of rectangular distributions.
\newblock {\em Statist. Probab. Lett.}, 3(2):107--109, 1985.

\bibitem[Ter83]{TerrellMR704575}
George~R. Terrell.
\newblock A characterization of rectangular distributions.
\newblock {\em Ann. Probab.}, 11(3):823--826, 1983.

\bibitem[Wit75]{WitsenhausenMR363678}
H.~S. Witsenhausen.
\newblock On sequences of pairs of dependent random variables.
\newblock {\em SIAM J. Appl. Math.}, 28:100--113, 1975.

\bibitem[Yan12]{YangMR2962669}
Bicheng Yang.
\newblock Hilbert-type integral operators: norms and inequalities.
\newblock In {\em Nonlinear analysis}, volume~68 of {\em Springer Optim.
  Appl.}, pages 771--859. Springer, New York, 2012.

\bibitem[Yu08]{YuMR2422962}
Yaming Yu.
\newblock On the maximal correlation coefficient.
\newblock {\em Statist. Probab. Lett.}, 78(9):1072--1075, 2008.

\end{thebibliography}

\end{document}